\documentclass[final]{siamltex}

\usepackage{graphicx}
\usepackage{amsmath}
\usepackage{makecell}
\usepackage{epsfig}
\usepackage{amsmath, amssymb}
\usepackage {graphicx}
\usepackage{float}
\usepackage{amsmath, amssymb}
\usepackage{mathptmx}      
\usepackage{multirow}
\usepackage{amssymb}
\usepackage{amsfonts}
\usepackage{mathrsfs}
\usepackage{amsmath, amssymb}
\usepackage{array}

\newtheorem{remark}{Remark}[section]
\newtheorem{algorithm}{Algorithm}

\title{A Kaczmarz Method with Simple Random Sampling for Solving Large Linear Systems}

\author{Yutong Jiang\thanks{School of Mathematics, China University of Mining and Technology, Xuzhou, 221116,  Jiangsu, P.R. China. E-mail: {\tt jiangyutong@qq.com}.}
       \and Gang Wu\thanks{Corresponding author. School of Mathematics,
China University of Mining and Technology, Xuzhou, 221116, Jiangsu, P.R. China.
E-mail: {\tt gangwu@cumt.edu.cn} and {\tt gangwu76@126.com}. This author is
supported by the Fundamental Research Funds for the Central Universities under grant 2019XKQYMS89.} \and Long Jiang\thanks{School of Mathematics, China University of Mining and Technology, Xuzhou, 221116,  Jiangsu, P.R. China. E-mail: {\tt jianglong365@hotmail.com}.}
}

\begin{document}
\maketitle

\begin{abstract}
The Kaczmarz method is a popular iterative scheme for solving large, consistent system of over-determined
linear equations. This method has been widely
used in many areas such as reconstruction of CT scanned images, computed tomography and signal processing.
In the Kaczmarz method, one cycles
through the rows of the linear system and each iteration is formed by projecting the
current point to the hyperplane formed by the active row. However, the Kaczmarz method may converge very slowly in practice.
The randomized Kaczmarz method (RK) greatly improves the convergence rate of the Kaczmarz method, by using the
rows of the coefficient matrix in random order rather than in their given order.
An obvious disadvantage of the randomized Kaczmarz method is its probability
criterion for selecting the active or working rows in the coefficient matrix.
In [{\sc Z.Z. Bai, W. Wu}, {\em On greedy randomized Kaczmarz method for solving large sparse linear systems}, SIAM Journal on Scientific Computing, 2018, 40: A592--A606], the authors
proposed a new
probability criterion that can capture larger entries of the residual vector of the linear
system as far as possible at each iteration step, and presented a greedy randomized Kaczmarz method (GRK).
However, the greedy Kaczmarz method may suffer from
heavily computational cost when the size of the matrix is large, and the overhead will be prohibitively
large for big data problems.
The contribution of this work is as follows.
First, from the probability significance point of view, we present a partially randomized Kaczmarz method, which can reduce the computational overhead needed in greedy randomized Kaczmarz method.
Second, based on Chebyshev's law of large numbers and Z-test, we apply a simple sampling approach to the partially randomized Kaczmarz method, and propose a randomized Kaczmarz method with simple random sampling for large linear systems. The convergence of the proposed method is established. Third, we apply the new strategy to the ridge regression problem, and propose a partially randomized Kaczmarz method with simple random sampling for ridge regression. Numerical experiments show numerical behavior of the proposed algorithms, and demonstrate their superiority over many state-of-the-art randomized Kaczmarz methods for large linear systems problems and ridge regression problems.

\end{abstract}

\pagestyle{myheadings} \thispagestyle{plain} \markboth{Y. JIANG, G. WU AND L. JIANG}{\sc A Randomized Kaczmarz Method with Simple Random Sampling}

\date{ }%

\section{Introduction}
\setcounter{equation}{0}

The Kaczmarz method \cite{bibitem12} is an effective algorithm for solving large consistent overdetermined linear system
\begin{equation}\label{1.1}
A{\bf x}={\bf b},
\end{equation}
where $A\in \mathbb{C}^{m\times n}$ and ${\bf b}\in \mathbb{C}^{m}$, with $m\geqslant n$.
This method has been applied to many important fields such as reconstruction of CT scanned images \cite{Had}, biological calculation \cite{bibitem20}, computerized
tomography \cite{Gow, bibitem10, bibitem19}, digital signal processing \cite{bibitem9,bibitem10,bibitem19}, image reconstruction \cite{Car, Li, Ram, Tho}, distributed computing \cite{Hef, bibitem23}; and so on \cite{Loe, Gua, Int, Bou, Nec,Xu}.

For the linear system \eqref{1.1}, the Kaczmarz method cycles
through the rows of the matrix in question, and each iteration is formed by projecting the
current point to the hyperplane formed by the active row. That is,
\begin{equation}\label{2.1}
{\bf x}_{k+1}={\bf x}_k+\frac{{\bf b}_{\left ( ik \right )}-A_{\left ( ik \right )}{\bf x}_k}{\left \| A_{\left ( ik \right )} \right \|_{2}^{2}}\left ( A_{\left ( ik \right )} \right )^{H},\quad ik = (k~mod~m) + 1,
\end{equation}
where ${\bf x}_k$ denotes the approximation obtained from the $k$-th iteration, $A_{(ik)}$ denotes the $i$-th row of the matrix $A$ in the $k$-th iteration, and ${\bf b}_{(ik)}$ stands for the $i$-th entry of the vector ${\bf b}$ in the $k$-th iteration, respectively.


Indeed, the Kaczmarz method tries to find the solution by successively projecting the current iteration solution ${\bf x}_{k}$ onto the hyperplane $H_{ik}=\left \{ {\bf x} |\left \langle A_{\left ( ik \right )}, ~{\bf x} \right \rangle={\bf b}_{(ik)} \right \}$.
Although Kaczmarz method has been proposed for a long time and widely used in many practical problems, its convergence rate is difficult to determine.
In \cite{bibitem8}, Strohmer and Vershynin proposed a randomized Kaczmarz (RK) method with expected exponential rate of convergence.
The idea is that using the
rows of the coefficient matrix in random order rather than in their given order,
can improve the convergence of the classical Kaczmarz method.

The randomized
Kaczmarz method is quite appealing for practical applications, and it has been received great attention by many researchers. To name a few,
Hefny {\it et al.} presented variants of the randomized Kaczmarz (RK) and randomized
Gauss-Siedel (RGS) for solving linear system and derive their convergence rates \cite{Hef}. Based on the randomized Kaczmarz method, Zouzias and Freris \cite{Zou} proposed a randomized extended Kaczmarz method for solving least squares problem. A sampling Kaczmarz-Motzkin method was proposed in \cite{Loe}, which is a combination of the Kaczmarz and Motzkin methods. An accelerated randomized Kaczmarz algorithm was presented in \cite{Liu}, and a general randomized sparse Kaczmarz
method for generating sparse approximate solutions to linear systems was proposed in \cite{Lei}. Some block Kaczmarz methods were investigated in \cite{Nec,bibitem6,bibitem7}. The convergence rate of the randomized Kaczmarz method was considered in \cite{bibitem1}.
One is recommended to see \cite{Du,Ma,bibitem5,Tho,bibitem11} and the references therein.


As was pointed out in \cite{bibitem3}, one weakness of the randomized Kaczmarz method is its probability
criterion for selecting the working rows in the coefficient matrix.
In order to deal with this problem,
Bai and Wu \cite{bibitem3} introduced
an effective probability criterion for selecting the working rows from the coefficient matrix, and construct a greedy randomized Kaczmarz method (GRK).
The GRK method takes into account the residuals during iterations as the probability selection criteria, and extracts the rows corresponding to large residuals.
Recently, some relaxed greedy randomized Kaczmarz methods were proposed in \cite{bibitem2,bibitem18}, and the strategy was generalized to solving the ridge regression problem \cite{bibitem14}.

However, both the greedy Kaczmarz method and their relaxed version may suffer from heavily computational cost. Indeed, in order to construct an index set in each step, one has to scan the residual vector from scratch during each iteration. This is unfavorable when the size of the matrix is large, especially for big data problems. Moreover, in the relaxed greedy randomized Kaczmarz method, it is required to choose relaxation parameters in advance, whose optimal value is difficult to determine in advance.

So as to overcome these difficulties, we try to propose new  probability
criterion for selecting the working rows in the randomized Kaczmarz method, and to improve the performance of the greedy Kaczmarz method.
From the probability significance point of view, we first present a partially randomized Kaczmarz method. In this method, we only need to seek the row with the largest (relative) residual in all the rows of the coefficient matrix, and there is no need to construct index sets as in the greedy Kaczmarz method. Thus, the partially randomized Kaczmarz method is (much) cheaper than the greedy Kaczmarz method. However, the partially randomized Kaczmarz method may still be time consuming for big data problems.
Thus, based on Chebyshev's law of large numbers and Z-test, we then apply a simple sampling approach to the partially randomized Kaczmarz method, and propose a randomized Kaczmarz method with simple random sampling for large linear systems. This method can reduce the computational overhead of the partially randomized Kaczmarz method significantly.
The convergence of the proposed methods are analyzed. Furthermore, we apply the new strategies to the ridge regression problem, and propose a partially randomized Kaczmarz method with simple random sampling for ridge regression.


The paper is organized as follows. In section 2, we briefly introduce randomized Kaczmarz method and the greedy randomized Kaczmarz method. In section 3, we propose a partially randomized Kaczmarz method and a a partially randomized Kaczmarz method with simple sampling. Theoretical results are given to show the rationality and feasibility of our proposed algorithms. In section 4, we apply our new strategies to the ridge regression problem. In Section 5, extensive numerical experiments are performed on some real-world problems and synthetic data sets. They demonstrate the numerical behavior of the proposed algorithms, and illustrate the superiority of the new algorithms over many state-of-the-art randomized Kaczmarz methods for large linear systems and ridge regression problems. Some concluding remarks are given in Section 6.



\section{The randomized Kaczmarz method and the greedy randomized Kaczmarz method}
\setcounter{equation}{0}

In stead of using the rows of the matrix $A$ in their given order, the randomized Kaczmarz method makes use of the following  probabilistic criterion for choosing working rows:
\begin{equation}\label{2.2}
pr(ik)=\frac{\left \| A_{(ik)}\right \|_{2}^{2}}{\left \| A \right \|_{F}^{2}},
\end{equation}
which is based on the norm of the row of the matrix $A$.
It was shown that the randomized Kaczmarz method can greatly improve the convergence rate of the Kaczmarz method. The algorithm is listed as follows, for more details, refer to \cite{bibitem8}.
\begin{algorithm}\label{alg1}
{\bf The Randomized Kaczmarz Method~~{(RK)}}~{\rm\cite{bibitem8}}\\
{{\bf Input:} $A$, $\bf b$, $l$ and ${\bf x}_0$;}\\
{{\bf Output:} The approximate solution $\widetilde{\bf x}$;}\\
{\bf 1.} for $k=0, 1, \ldots, l-1$ do\\
{\bf 2.} Select $ik\in\left \{1, 2, \ldots, m  \right \}$ with probability $pr(row=ik)=\frac{\left \| A_{(ik)}
\right \|_{2}^{2}}{\left \| A \right \|_{F}^{2}}$;\\
{\bf 3.} Let ${\bf x}_{k+1}={\bf x}_k+\frac{{\bf b}_{\left ( ik \right )}-A_{\left ( ik \right )}{\bf x}_k}{\left \| A_{\left ( ik \right )} \right \|_{2}^{2}}\left ( A_{\left ( ik \right )} \right )^{H}$. If ${\bf x}_{k+1}$ is accurate enough, then stop, else continue;\\
{\bf 4.} endfor
\end{algorithm}

Moreover, Strohmer and Vershynin proved the following result, showing that the randomized Kaczmarz method converges with expected exponential
rate of convergence, and the convergence speed is closely related to the condition number of $A$.
\begin{theorem} \cite{bibitem8}
 Let ${\bf x}$ be the solution of \eqref{1.1}. Then Algorithm \ref{alg1} converges to ${\bf x}$ in expectation, with the average error
\begin{equation}
E \left \| {\bf x}_k-{\bf x} \right \|_{2}^{2}\leq \left ( 1-\kappa^{-2} \left ( A \right ) \right )^{k}\left \| {\bf x}_0-{\bf x} \right \|_{2}^{2},
\end{equation}
where $\kappa \left ( A \right )^{-2}:=\left \| A \right \|_{F}^{-2}\left \| A^{\dag} \right \|_{2}^{-2}$, and $ A^{\dag}$ is the left inverse of $A$.
\end{theorem}

Indeed, the randomized Kaczmarz method
is convergent in expectation to the unique
least-norm solution of the linear system \eqref{1.1}, when the coefficient matrix $A$ is of full column rank with $m\geq n$ or is of full row rank with
$m\leq n$ \cite{bibitem8,Ma}. Specifically, when  the linear system \eqref{1.1} is consistent, it was shown that the iteration sequence $\{{\bf x}_k\}$ converges to the unique least-norm solution $A^{\dag}{\bf b}$ \cite{Gow}.



An obvious disadvantage of the randomized Kaczmarz method is its probability
criterion for selecting the active or working rows in the coefficient matrix. For instance, in the unitary matrix, all the row norms of the matrix are the same, and one has to choose the working rows arbitrarily.
Another example is the coefficient matrix of the form $A={\rm diag}(1,10^8)$. In this case, the probability of choosing the first row is almost zero, as the norm of the second row is much larger than that of the first one.


 More precisely, in the $k$-th iteration, let the residual vector be ${\bf r}_{k}={\bf b}-A{\bf x}_{k}$, if $|{\bf r}_{\left(ik\right)}|>|{\bf r}_{\left(jk\right)}|,i,j\in\{1,2,\ldots,m\}$, then the probability of choosing the $i$-th row as the working row in the $(k+1)$-th iteration will be larger than that of the $j$-th row, where ${\bf r}_{\left(ik\right)}$ and ${\bf r}_{\left(jk\right)}$ represent the $i$-th and the $j$-th elements of ${\bf r}_{k}$, respectively. The randomized Kaczmarz method is described as follows, for more details and its implementations, refer to \cite{bibitem3}:
\begin{algorithm}\label{alg2}
{\bf The Greedy Randomized Kaczmarz Method {(GRK)}}~{\rm\cite{bibitem3}}\\
{{\bf Input:} $A$, $\bf b$, $l$ and $\bf{x_0}$;} \\
{{\bf Output:} The approximate solution $\widetilde{\bf x}$;}\\
{\bf 1.} for $k=0, 1,\ldots, l-1$ do\\
{\bf 2.} Compute
\begin{equation}\label{2.4}
\epsilon_k=\frac{1}{2}\left ( \frac{1}{\left \| {\bf b}-A{\bf x}_k \right \|_{2}^{2}} \max_{1\leq i\leq m}\left \{ \frac{\left | {{\bf b}_{\left ( ik \right )}-A_{\left ( ik \right )}{\bf x}_{k}} \right |^{2}}{\left \| A_{\left ( ik \right )} \right \|_{2}^{2}}\right \}+\frac{1}{\left \| A \right \|_{F}^{2}}\right )
\end{equation}\\
{\bf 3.} Determine the index set of positive integers
\begin{equation}\label{2.5}
\upsilon _k=\left \{ ik\Big| | {{\bf b}_{\left ( ik \right )}-A_{\left ( ik \right )}{\bf x}_{k}} |^{2}\geqslant \epsilon _k\left \| b-A{\bf x}_k \right \|_{2}^{2}\left \| A_{\left ( ik \right )} \right \|_{2}^{2} \right \}
\end{equation}\\
{\bf 4.} Compute the $i$-th entry $\widetilde{\bf r}_{\left ( ik \right )}$ of the vector $\widetilde{\bf r}_{k}$ according to
$$
\widetilde{\bf r}_{\left ( ik \right )}=\begin{cases}
{\bf b}_{\left ( i \right )}-A_{\left ( i \right )}{\bf x}_{k}, & \text{ if } i\in\upsilon _k \\
0, & \text{ otherwise }
\end{cases}
$$\\
{\bf 5.} Select $ik\in\upsilon _k$ with probability:
\begin{equation}\label{2.6}
pr(ik)=\frac{\widetilde{\bf r}_{(ik)}^{2}}{\left \| \widetilde{\bf r}_{k} \right \|_{2}^{2}}
\end{equation}
\\
{\bf 6.} Let ${{\bf x}}_{k+1}={{\bf x}}_k+\frac{{{\bf b}}_{\left ( ik \right )}-A_{\left ( ik \right )}{{\bf x}}_k}{\left \| A_{\left ( ik \right )} \right \|_{2}^{2}}\left ( A_{\left ( ik \right )} \right )^{H}$. If ${\bf x}_{k+1}$ is accurate enough, then stop, else continue;\\
{\bf 7.} endfor
\end{algorithm}

The convergence property of the greedy randomized Kaczmarz method was established in \cite{bibitem3}, and the main result is given as follows:
\begin{theorem} {\rm\cite{bibitem3}}
Let ${\bf x}$ be the solution of \eqref{1.1}. Then Algorithm \ref{alg2} converges to ${\bf x}$ in expectation, with
\begin{equation}
\mathbb{E}\left \| {\bf x}_{1}-{\bf x} \right \|_{2}^{2}\leq \left[ 1-\frac{\lambda_{\min}\left( A^{H}A \right)}{\left\| A \right\|_{F}^{2}}\right]\left \| {\bf x}_{0}-{\bf x} \right \|_{2}^{2}, ~~k=0
\end{equation}

and

\begin{equation}
\mathbb{E}_k\left \| {\bf x}_{k+1}-{\bf x} \right \|_{2}^{2}\leq \left[ 1-\frac{1}{2}\Big ( \frac{1}{\gamma }\left \| A \right \|_{F}^{2}+1 \Big)\kappa^{-2} \left ( A \right )\right]\left \| {\bf x}_k-{\bf x} \right \|_{2}^{2},~~k=1,2,\ldots
\end{equation}
where $\gamma =\max_{1\leq i\leq m}\sum_{j=1,j\neq i}^{m}\left \| A_{\left ( j \right )} \right \|_{2}^{2}$.
\end{theorem}

As $\frac{1}{2}( \frac{1}{\gamma }\| A \|_{F}^{2}+1)>1$, the convergence factor of the greedy randomized Kaczmarz method is smaller than that of the randomized Kaczmarz. Hence, the greedy randomized Kaczmarz method would converge faster than the randomized Kaczmarz method.
On the basis of GRK method, Bai and Wu \cite{bibitem2}
further generalize the greedy randomized Kaczmarz method via introducing a relaxation
parameter in the involved probability criterion, and propose a class of relaxed greedy
randomized Kaczmarz methods (RGRK).
The key is that the greedy factor $\epsilon_k$ used in the RRK method is different from the one used in Algorithm \ref{alg2} \big(refer to \eqref{2.4}\big), and it is chosen as
\begin{equation}\label{288}
\epsilon_k=\frac{\theta}{\left \|{\bf b}-A{\bf x}_k \right \|_{2}^{2}} \max_{1\leq i\leq m}\left \{ \frac{\left | {{\bf b}_{\left ( ik \right )}-A_{\left ( ik \right )}{\bf x}_{k}} \right |^{2}}{\left \| A_{\left ( ik \right )} \right \|_{2}^{2}}\right \}+\frac{1-\theta}{\left \| A \right \|_{F}^{2}},
\end{equation}
where $0\leq\theta\leq 1$ is a user-provided parameter. Obviously, the relaxed greedy
randomized Kaczmarz method reduces to the greedy
randomized Kaczmarz method as $\theta=\frac{1}{2}$.
The main convergence result on the relaxed greedy
randomized Kaczmarz method is listed as follows:
\begin{theorem} {\rm \cite{bibitem2}}\label{Thm2.3}
Let ${\bf x}$ be the solution of \eqref{1.1}. Then the relaxed greedy
randomized Kaczmarz method converges to ${\bf x}$ in expectation, with
\begin{equation}
\mathbb{E}\left \| {\bf x}_{1}-{\bf x} \right \|_{2}^{2}\leq \left[ 1-\frac{\lambda_{\min}\left( A^{H}A \right)}{\left\| A \right\|_{F}^{2}}\right]\left \| {\bf x}_{0}-{\bf x} \right \|_{2}^{2}, ~~k=0
\end{equation}

and

\begin{equation}\label{299}
\mathbb{E}_k\left \| {\bf x}_{k+1}-{\bf x} \right \|_{2}^{2}\leq \left[1-\Big( \frac{\theta}{\gamma }\left \| A \right \|_{F}^{2}+\big(1-\theta\big)\Big)\kappa^{-2} \left ( A \right )\right]\left \| {\bf x}_k-{\bf x} \right \|_{2}^{2}, ~~k=1,2,\ldots
\end{equation}
where $\gamma =\max_{1\leq i\leq m}\sum_{j=1, j\neq i}^{m}\left \| A_{\left ( j \right )} \right \|_{2}^{2}$.
\end{theorem}


However, both the greedy Kaczmarz method and the relaxation method may suffer from large overhead in practice \cite{bibitem2,bibitem3,bibitem18}. More precisely, to determine the index set $\upsilon _k$ defined in \eqref{2.5}, in the algorithms we have to scan the residual vector from scratch in each iteration. This is unfavorable when the size of the matrix is large, and the overhead will be prohibitively large for big data problems. Furthermore, the relaxed greedy
randomized Kaczmarz methods are parameter-dependent \cite{bibitem2,bibitem18}, and the optimal parameters are difficult to choose in advance.
Therefore, it is urgent to investigate new probability
criteria for selecting the working rows for the greedy Kaczmarz method and its variants, so that one can further speed up the convergence of the  Kaczmarz-type methods.


\section{A Randomized Kaczmarz Method with Simple Sampling}\label{Sec3}
\setcounter{equation}{0}
The randomized Kaczmarz method only adopts a probability criterion that is determined
by the ratio between the
2-norms of the rows of $A$ and $\|A\|_F$.
As a comparison, the probability criterion used by the greedy Kaczmarz method
is essentially determined by two factors: one is the largest entry of the
residual with respect to the current iterate, and the other is the ratio between the
norms of some rows of the coefficient matrix and the coefficient matrix itself \cite{bibitem3}; see \eqref{2.4}--\eqref{2.6}.

In this section, we first present a new probability criterion for choosing working rows in the randomized Kaczmarz method, and show rationality of the proposed strategy. Second, based on Chebyshev's law of large numbers and Z-test, we propose a simple sampling approach for randomized Kaczmarz method. Third, we propose a Kaczmarz method with simple random sampling for large linear systems, which is the main algorithm of this paper, and discuss the convergence of the proposed method.

\subsection{A partially randomized Kaczmarz method}

Let ${\bf x}$ be the exact solution of the equation \eqref{1.1}, we have from \eqref{2.1} that
%
\begin{align}\label{31}
A_{\left ( ik \right )} \left ( {\bf x}_{k+1}-{\bf x} \right )&=A_{\left ( ik \right )} \left ({\bf x}_k-{\bf x}+\frac{{\bf{b}}_{\left ( ik \right )}-A_{\left ( ik \right )}{\bf x}_k}{\left \| A_{\left ( ik \right )} \right \|_{2}^{2}}\left ( A_{\left ( ik \right )} \right )^{H}  \right )\nonumber\\
&=A_{\left ( ik \right )} \left ( {\bf x}_{k}-{\bf x} \right )+\left ( {\bf{b}}_{\left ( ik \right )}-A_{\left ( ik \right )}{\bf x}_k \right )\nonumber\\
&={\bf{b}}_{\left ( ik \right )}-A_{\left ( ik \right )}{\bf x}\nonumber\\
&=0,\quad \quad k=0, 1, \ldots
\end{align}
That is to say, for the randomized Kaczmarz method, there holds \cite{bibitem8}
\begin{equation}\label{3.1}
\left \| {\bf x}_{k+1}-{\bf x} \right \|_{2}^{2}=\left \| {\bf x}_{k}-{\bf x} \right \|_{2}^{2}-\left \| {\bf x}_{k+1}-{\bf x}_k \right \|_{2}^{2}.
\end{equation}
Note that ${\bf x}_k$ is measurable, taking conditional expectations on both sides of the above equality gives
\begin{equation*}\label{f3.2}
\mathbb{E}_{k}\left \| {\bf x}_{k+1}-{\bf x} \right \|_{2}^{2}=\left \| {\bf x}_{k}-{\bf x} \right \|_{2}^{2}-\mathbb{E}_{k}\left \| {\bf x}_{k+1}-{\bf x}_k \right \|_{2}^{2}.
\end{equation*}
However, ${\bf x}_{k+1}$ is not only related to the previous approximation ${\bf x}_k$, but also to all the predecessors $\{{\bf x}_v\}_{v=0}^k$.
Taking {\it expectation} instead of {\it conditional expectations} on both sides of \eqref{3.1} yields
\begin{equation}\label{f3.3}
E\left \| {\bf x}_{k+1}-{\bf x} \right \|_{2}^{2}=E\left \| {\bf x}_{k}-{\bf x} \right \|_{2}^{2}-E\left \| {\bf x}_{k+1}-{\bf x}_k \right \|_{2}^{2},
\end{equation}
which can be rewritten as
\begin{equation}\label{f3.4}
{E}\left \| {\bf x}_{k+1}-{\bf x} \right \|_{2}^{2}=\left(1-\frac{{E}\left \| {\bf x}_{k+1}-{\bf x}_k \right \|_{2}^{2}}{{E}\left \| {\bf x}_{k}-{\bf x} \right \|_{2}^{2}}\right){E}\left \| {\bf x}_{k}-{\bf x} \right \|_{2}^{2}.
\end{equation}

By using \eqref{f3.3} recursively, we get
\begin{align*}
{E}\left \|{\bf x}_{k}-{\bf x} \right \|_{2}^{2}&={E}\left \| {\bf x}_{k-1}-{\bf x} \right \|_{2}^{2}-{E}\left \| {\bf x}_{k}-{\bf x}_{k-1} \right \|_{2}^{2}\\
&={E}\left \|{\bf x}_{k-2}-{\bf x} \right \|_{2}^{2}-{E}\left \| {\bf x}_{k}-{\bf x}_{k-1} \right \|_{2}^{2}-{E}\left \| {\bf x}_{k-1}-{\bf x}_{k-2} \right \|_{2}^{2}\\
&=\cdots\\
&={E}\left \|{\bf x}_{0}-{\bf x} \right \|_{2}^{2}-{E}\left \| {\bf x}_{0}-{\bf x}_{1} \right \|_{2}^{2}-\cdots-\left \| {\bf x}_{k}-{\bf x}_{k-1} \right \|_{2}^{2}\\
&=\left \|{\bf x}_{0}-{\bf x} \right \|_{2}^{2}-\sum_{v=1}^{k-1}{E}\left \| {\bf x}_{v-1}-{\bf x}_{v} \right \|_{2}^{2}.
\end{align*}
As a result, \eqref{f3.4} can be reformulated as
\begin{equation}\label{f3.5}
{E}\left \| {\bf x}_{k+1}-{\bf x} \right \|_{2}^{2}=\left(1-\frac{{E}\left \| {\bf x}_{k+1}-{\bf x}_k \right \|_{2}^{2}}{\left \| {\bf x}_{0}-{\bf x} \right \|_{2}^{2}-\sum_{v=1}^{k-1}{E}\left \| {\bf x}_{v-1}-{\bf x}_{v} \right \|_{2}^{2}}\right){E}\left \| {\bf x}_{k}-{\bf x} \right \|_{2}^{2}.
\end{equation}

So far, we have established the relationship between ${E}\left \| {\bf x}_{k+1}-{\bf x} \right \|_{2}^2$ and ${E}\left \| {\bf x}_{k}-{\bf x} \right \|_{2}^2$. By \eqref{f3.5}, it is seen that the convergence speed of the randomized Kaczmarz method is closely related to the expectation of $\|{\bf x}_{k+1}-{\bf x}_k\|_2^2$ and those of $\|{\bf x}_{v-1}-{\bf x}_{v}\|_2^2,~v=1,2,\ldots,k$, and the larger
\begin{equation}\label{3.5}
{E}\left \| {\bf x}_{v-1}-{\bf x}_{v} \right \|_{2}^{2}={\sum_{i=1}^{m} pr\left(iv\right)\frac{{{\bf r}_{iv}}^{2}}{\left\|A_{(iv)}\right\|_{2}^{2}} },\quad v=1,2,\ldots,k+1,
\end{equation}
the faster the convergence speed will be, where ${\bf r}_{iv}$ is the $i$-th element of the vector ${\bf r}_v={\bf b}_{v}-A{\bf x}_{v}$. We need the following classical inequality before discussing the choice of the probabilities $\{pr\left(iv\right)\}_{i=1}^m$.
\begin{theorem}{\cite{bibitem22}}\label{Thm3.1}
Given two sequences
\begin{equation}\label{37}
0\leq c_{1}\leq c_{2}\leq\cdots\leq c_{m}\quad and \quad 0\leq p_{1}\leq p_{2}\leq\cdots\leq p_{m}\leq 1.
\end{equation}
Let $\left\{\widetilde{p}_{1},\widetilde{p}_{2},\ldots,\widetilde{p}_{m}\right\}$ be any rearrangement of the set $\left\{{p}_{1},{p}_{2},\ldots,{p}_{m}\right\}$. Then $E(c)=\sum\limits_{\substack{i=1}}^{m}\widetilde{p}_{i}c_{i}$ reaches the maximal value
if and only if $\widetilde{p}_{i}={p}_{i},~i=1,2,\ldots,m$.
\end{theorem}

Given the set of values $\left \{ c_i\right \}_{i=1}^m$ and the set of probabilities $\left \{ p_i\right \}_{i=1}^m$, the probabilities can be arranged in any order $\left \{\widetilde{p}_i\right \}_{i=1}^m$ theoretically.
Theorem \ref{Thm3.1} shows that, if the given sequences $\{c_{i}\}$'s and $\{p_{i}\}$'s share the same order,
then $\sum\limits_{i=1}^{m}p_{i}c_{i}$ will reach the maximal value for arbitrary disordered arrangements on the $\{p_{i}\}$'s. Specifically, if we denote by $c_{i}={{\bf r}_{ik}^2}/{\left\|A_{ik}\right\|_{2}^{2}},~1\leq i\leq m$, and choose the probabilities as
$$
pr(ik)=\frac{\frac{{\bf r}_{ik}^{2}}{\left\|A_{ik}\right\|^{2}_{2}}}{\sum_{i=1}^{m} \frac{{{\bf r}_{ik}}^{2}}{\left\|A_{ik}\right\|_{2}^{2}} },
$$
then the $\{p_i\}$'s are nothing but the strategy used in the greedy randomized Kaczmarz method; refer to \eqref{2.6}. 

Further, we can use the probabilities as follows
\begin{equation}\label{38}
pr\left ( ik \right )=\frac{\frac{\left|{\bf r}_{ik}\right|^{t}}{\left\|A_{ik}\right\|^{t}_{2}}}{\sum_{i=1}^{m} \frac{{\left|{\bf r}_{ik}\right|}^{t}}{\left\|A_{ik}\right\|_{2}^{t}}},\quad i=1,2,\ldots,m,
\end{equation}
where $t\geq 1$ is a positive integer. Notice that this choice also satisfies the condition \eqref{37}.
Moreover, the larger the parameter $t$, the higher {\it the probability significance} and the larger the ${E}\left\|{\bf x}_{k+1}-{\bf x}_{k}\right\|^{2}_{2}$. Indeed, we have
\begin{equation}\label{399}
\lim\limits_{t\to\infty}\frac{\max\limits_{1\leq i\leq m}\left (\frac{\left|{\bf r}_{ik}\right|^{t}}{\left\|A_{ik}\right\|^{t}_{2}}  \right )}{\sum\limits_{i=1}^{m} \frac{{\left|{\bf r}_{ik}\right|}^{t}}{\left\|A_{ik}\right\|_{2}^{t}} }=\lim\limits_{t\to\infty}\frac{1}{{{\sum\limits_{i=1}^{m} \left(\frac{{\left|{\bf r}_{ik}\right|}^{t}}{\left\|A_{ik}\right\|_{2}^{t}}/\max\limits_{1\leq i\leq m}\left (\frac{\left|{\bf r}_{ik}\right|^{t}}{\left\|A_{ik}\right\|^{t}_{2}}  \right )\right)}}}=1.
\end{equation}
With the probabilities $pr\left ( ik \right )$ defined in \eqref{38}, we can present the following algorithm. The key is that the rows corresponding to the current maximum homogenization residuals are selected.
\begin{algorithm}\label{alg4}
{\bf A randomized Kaczmarz method with residual homogenizing}\\
{{\bf Input:} $A$, $\bf b$, $t\geq 1$, and ${\bf x}_0$, as well as the maximal iteration number $l$;}\\
{{\bf Output:} The approximate solution $\widetilde{\bf x}$;}\\
{\bf 1}. for $k=0, 1, \ldots, l-1$ do\\
{\bf 2}. Select $ik$ with probability $pr\left ( row=ik \right )=\frac{\frac{\left|{\bf r}_{ik}\right|^{t}}{\left\|A_{ik}\right\|^{t}_{2}}}{\sum_{i=1}^{m} \frac{{\left|{\bf r}_{ik}\right|}^{t}}{\left\|A_{ik}\right\|_{2}^{t}}}$;\\
{\bf 3}. Let ${\bf x}_{k+1}={\bf x}_k+\frac{{\bf{b}}_{\left ( ik \right )}-A_{\left ( ik \right )}{\bf x}_k}{\left \| A_{\left ( ik \right )} \right \|_{2}^{2}}\left ( A_{\left ( ik \right )} \right )^{H}$. If ${\bf x}_{k+1}$ is accurate enough, then stop, else continue;\\
{\bf 4.} endfor
\end{algorithm}

\begin{figure}[H]\label{Fig3.1}
\begin{minipage}{0.4\linewidth}
  \centerline{\includegraphics[width=6cm]{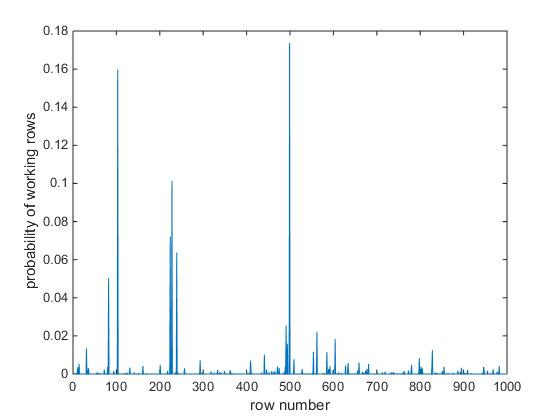}}
  \centerline{(a) t=2}
\end{minipage}
\hfill
\begin{minipage}{0.4\linewidth}
  \centerline{\includegraphics[width=6cm]{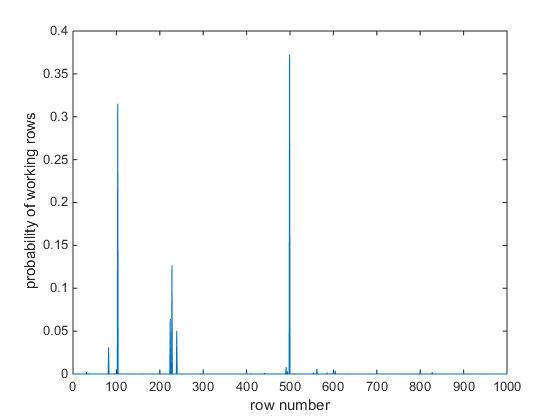}}
  \centerline{(b) t=4}
\end{minipage}
\vfill
\begin{minipage}{0.4\linewidth}
  \centerline{\includegraphics[width=6cm]{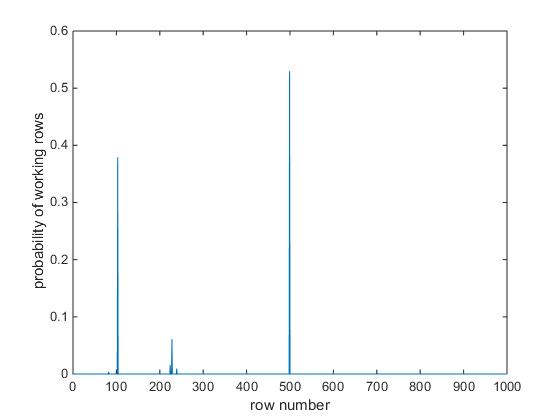}}
  \centerline{(c) t=8}
\end{minipage}
\hfill
\begin{minipage}{0.4\linewidth}
  \centerline{\includegraphics[width=6cm]{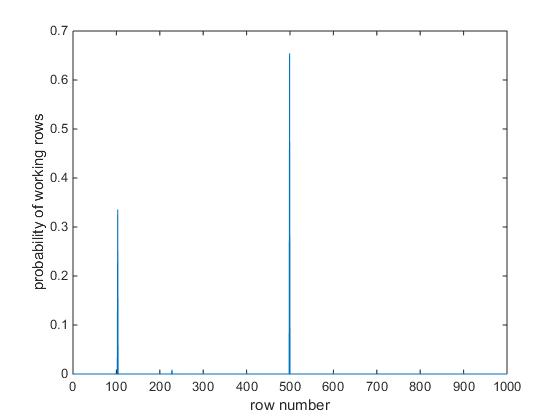}}
  \centerline{(d) t=16}
\end{minipage}
\caption{Probabilities of the working rows in Algorithm \ref{alg4} (the first iteration) with different $t=2,4,8,16$. The coefficient matrix $A$ is randomly generated by using the MATLAB command {\tt randn(1000,100)}.}
\label{fig:res}
\end{figure}
%

\begin{figure}[h]

  \centerline{\includegraphics[width=14.0cm]{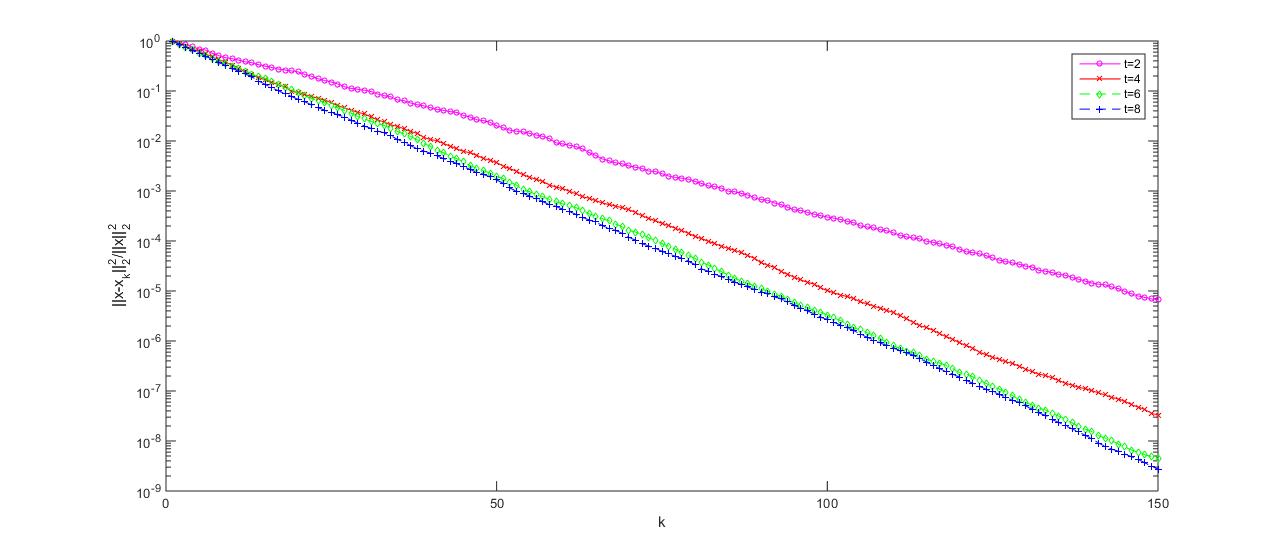}}\label{Fig32}

\caption{Convergence curves of Algorithm \ref{alg4} with $t=2,4,6,8$. The coefficient matrix $A$ is randomly generated by using the MATLAB command {\tt randn(1000,100)}.}
\label{fig:rres}
\end{figure}
%

To illustrate the rationality of using \eqref{38} more precisely, we plot in Figure \ref{Fig3.1} the probabilities of the working rows of Algorithm \ref{alg4} (the first iteration) with different $t=2,4,8,16$, where the coefficient matrix $A$ is randomly generated by using the MATLAB command {\tt randn(1000,100)}. It is seen that the rows with larger probabilities are easily accessible as $t$ increases.
In Figure 3.2, we plot the convergence curves of Algorithm \ref{alg4} with $t=2,4,6,8$. It is observed that the algorithm converges faster with a larger $t$.
Thus, a natural idea is to set $t\to\infty$ in Algorithm \ref{alg4}.

According to \eqref{399}, the probability of choosing the $ik$-th row such that
\begin{equation}\label{39}
\frac{|{\bf r}_{(ik)}|}{\| A_{(ik)}\|_{2}}=\frac{\left | {{\bf{b}}_{\left ( ik \right )}-A_{\left ( ik \right )}{\bf x}_{k}} \right |}{\left \| A_{\left ( ik \right )} \right \|_{2}}=\max_{1\leq j\leq m}\left \{ \frac{\left | {{\bf{b}}_{\left ( jk \right )}-A_{\left ( jk \right )}{\bf x}_{k}} \right |}{\left \| A_{\left ( jk \right )} \right \|_{2}}\right \}
\end{equation}
is one. So we have the following algorithm.
\begin{algorithm}\label{alg41}
{\bf A partially randomized Kaczmarz method for linear systems {(PRK)}}\\
{{\bf Input:} $A$, $\bf b$, $l$ and $\bf{x_0}$;}\\
{{\bf Output:}  The approximate solution $\widetilde{\bf x}$;}\\
{\bf 1}. for $k=0,1,\ldots,l-1$ do\\
{\bf 2}. Select the working row number $ik$ according to \eqref{39};
\\
{\bf 3}. Let ${\bf{x}}_{k+1}={\bf{x}}_k+\frac{{\bf{b}}_{\left ( ik \right )}-A_{\left ( ik \right )}{\bf{x}}_k}{\left \| A_{\left ( ik \right )} \right \|_{2}^{2}}\left ( A_{\left ( ik \right )} \right )^{H}$. If ${\bf x}_{k+1}$ is accurate enough, then stop, else continue;\\
{\bf 4}. endfor
\end{algorithm}

Notice that this algorithm is no longer a random algorithm in the general sense, so we called it ``partially randomized" Kaczmarz method. On the other hand, recall that in the greedy randomized Kaczmarz method (GRK) and the relaxed greedy randomized Kaczmarz method (RGRK), one has to evaluate $\epsilon_k$ and determine the index set $\upsilon _k$ during each iteration, which is very time-consuming. As a comparison, there is no need to determine the index set $\upsilon _k$ anymore, and it is only required to find the row with the largest (relative) residual $|{\bf r}_{(ik)}|/\| A_{(ik)}\|_{2}$. Thus, Algorithm \ref{alg41} can reduce the computational overhead per iteration of GRK and RGRK significantly.

\begin{remark}
We point out that Algorithm \ref{alg41} is equivalent to the relaxed greedy randomized Kaczmarz method with the relaxation parameter $\theta=1$, making this method deteriorate
to a ``partially" randomized process; see \eqref{288}. However, in \cite[pp.24]{bibitem2}, Bai and Wu emphasize that $\theta=1$ is not a good choice for the relaxed greedy randomized Kaczmarz method. Here our contribution is to indicate that $\theta=1$ is a good choice indeed, from the probability significance point of view. This algorithm also appeared in \cite{bibitem23}, but we consider the original intention of this algorithm is different from \cite{bibitem23}. In this paper, we regard this algorithm as a special case of random method, and we get better convergence result.
\end{remark}

The following theorem shows the convergence of the partially randomized Kaczmarz method.
\begin{theorem}\label{Thm3.2}
Let ${\bf x}$ be the solution of \eqref{1.1}. Let
\begin{equation}\label{eqq312}
\gamma =\max_{1\leq i\leq m}\sum_{j=1, j\neq i}^{m}\left \| A_{\left ( j \right )} \right \|_{2}^{2},\quad{\rm and}\quad \xi=\left\| A \right\|_{F}^{2}-\left\| A_{\left(s\right) }\right\|_{2}^{2},
\end{equation}
where $A_{\left(s\right)}$ denotes the second smallest row of $A$ in 2-norm.
Then Algorithm \ref{alg41} converges to ${\bf x}$ in expectation, with
\begin{equation}\label{311}
\mathbb{E}_k\left \| {\bf x}_{k+1}-{\bf x} \right \|_{2}^{2}\leq \left[\Big (  1-\frac{\left \| A \right \|_{F}^{2}}{\gamma }\kappa^{-2}\left ( A \right ) \Big )\Big (  1-\frac{\left \| A \right \|_{F}^{2}}{\xi }\kappa^{-2}\left ( A \right ) \Big )\right]\left \| {\bf x}_{k-1}-{\bf x} \right \|_{2}^{2},\quad k\geq 1.
\end{equation}
\end{theorem}
\begin{proof}
Taking conditional expectation on both sides of \eqref{3.1} gives
\begin{align}\label{3.7}
\mathbb{E}_k\left \|{\bf x}_{k+1}-{\bf x} \right \|_{2}^{2}&=\left \| {\bf x}_{k}-{\bf x} \right \|_{2}^{2}-\mathbb{E}_k\left \| {\bf x}_{k+1}-{\bf x}_{k} \right \|_{2}^{2}\nonumber\\
&=\left \| {\bf x}_{k}-{\bf x} \right \|_{2}^{2}-\max_{1\leq j\leq m}\left \{ \frac{\left | {{\bf{b}}_{\left ( jk \right )}-A_{\left ( jk \right )}{\bf x}_{k}} \right |^{2}}{\left \| A_{\left ( jk \right )} \right \|_{2}^{2}}\right \},
\end{align}
where \eqref{3.7} follows from \eqref{39} and the probability of choosing the $ik$-th row is (almost) 1.
We notice from \eqref{31} that
\begin{align}\label{eq313}
{\bf r}_{(ik)}=A_{\left ( ik \right )} \left ( {\bf x}_{k+1}-{\bf x} \right )&=A_{\left ( ik \right )} \left ({\bf x}_{k}-{\bf x}+\frac{{\bf{b}}_{\left ( ik \right )}-A_{\left ( ik \right )}{\bf x}_{k}}{\left \| A_{\left ( ik \right )} \right \|_{2}^{2}}\left ( A_{\left ( ik \right )} \right )^{H}  \right )\nonumber\\
&=A_{\left ( ik \right )} \left ( {\bf x}_{k}-{\bf x} \right )+\left ( {\bf{b}}_{\left ( ik \right )}-A_{\left ( ik \right )}{\bf x}_{k} \right )\nonumber\\
&={\bf{b}}_{\left ( ik \right )}-A_{\left ( ik \right )}{\bf x}\nonumber\\
&=0.
\end{align}
In other words, the probability of choosing the $ik$-th row is zero.
Thus, if we denote by $C_{jk}={{\bf r}_{(jk)}^{2}}/{\left\|A_{jk}\right\|_{2}^{2}}$, then
\begin{align}\label{3.8}
\max_{1\leq j\leq m}\left \{ \frac{\left | {{\bf{b}}_{\left ( jk \right )}-A_{\left ( jk \right )}{\bf x}_{k}} \right |^{2}}{\left \| A_{\left ( jk \right )} \right \|_{2}^{2}}\right \}&=\frac{\max\limits_{1\leq j\leq m}\left \{ \frac{\left | {{\bf{b}}_{\left ( jk \right )}-A_{\left ( jk \right )}{\bf x}_{k}} \right |^{2}}{\left \| A_{\left ( jk \right )} \right \|_{2}^{2}}\right \}}{\left \| {\bf r}_{k} \right \|_{2}^{2}}\left \| {\bf r}_{k} \right \|_{2}^{2}\nonumber\\
&=\frac{\max\limits_{1\leq j\leq m}\left ( C_{jk} \right )}{\sum_{jk=1}^{m}{\left\|A_{\left(jk\right)}\right\|_2^2}C_{jk}}\left \| {\bf r}_{k} \right \|_{2}^{2}\nonumber\\
&\geq \frac{\left \| {\bf r}_{k} \right \|_{2}^{2}}{\sum_{jk=1}^{m}\left \| A_{\left(jk\right)} \right \|_{2}^{2}}
=\frac{\left \| {\bf r}_{k} \right \|_{2}^{2}}{\sum_{jk=1, jk\neq jk-1}^{m}\left \| A_{\left(jk\right)} \right \|_{2}^{2}}\nonumber\\
&\geq \frac{\lambda_{\min}\left (A^{H}A\right )}{\gamma }\left \| {\bf x}_{k}-{{\bf x}} \right \|_{2}^{2}.
\end{align}

So it follows from \eqref{3.7} and \eqref{3.8} that
\begin{eqnarray}
\mathbb{E}_k\left \|{\bf x}_{k+1}-{\bf x} \right \|_{2}^{2}&=\left \| {\bf x}_{k}-{\bf x} \right \|_{2}^{2}-\max_{1\leq j\leq m}\left \{ \frac{\left | {{\bf{b}}_{\left ( jk \right )}-A_{\left ( jk \right )}{\bf x}_{k}} \right |^{2}}{\left \| A_{\left ( jk \right )} \right \|_{2}^{2}}\right \}\nonumber\\
&\leq \left \| {\bf x}_{k}-{\bf x} \right \|_{2}^{2}-\frac{\lambda_{\min}\left (A^{H}A\right )}{\gamma }\left \| {\bf x}_{k}-{{\bf x}} \right \|_{2}^{2}\nonumber\\
&= \left ( 1-\frac{\left \| A \right \|_{F}^{2}}{\gamma }\kappa \left ( A \right )^{-2}\right )  \left \| {\bf x}_{k}-{\bf x} \right \|_{2}^{2}.\label{3.10}
\end{eqnarray}
Notice from \eqref{eq313} that Algorithm \ref{alg41} will not select the same row in two consecutive iterations. Thus,
\begin{eqnarray}
\mathbb{E}_{k-1}\left \|{\bf x}_{k}-{\bf x} \right \|_{2}^{2}&=\left \| {\bf x}_{k-1}-{\bf x} \right \|_{2}^{2}-\frac{{\bf r}_{(ik-1)}^2}{\sum_{jk-1=1, jk-1\neq jk-2}^{m}\big(\left\| A \right\|_{F}^{2}-\left\| A_{(jk-1)} \right\|_{2}^{2}\big)}\nonumber\\
&\leq \left \| {\bf x}_{k-1}-{\bf x} \right \|_{2}^{2}-\frac{\lambda_{\min}\left (A^{H}A\right )}{\xi }\left \| {\bf x}_{k-1}-{{\bf x}} \right \|_{2}^{2}\nonumber\\
&= \left ( 1-\frac{\left \| A \right \|_{F}^{2}}{\xi }\kappa \left ( A \right )^{-2}\right )  \left \| {\bf x}_{k-1}-{\bf x} \right \|_{2}^{2},
\label{3.11}
\end{eqnarray}
and a combination of \eqref{3.10} with \eqref{3.11} gives \eqref{311}.
\end{proof}

\begin{remark}
As $\|A\|_F\geq\gamma\geq\xi$, we have
\begin{equation}
\Big( 1-\frac{\left \| A \right \|_{F}^{2}}{\gamma }\kappa(A)^{-2} \Big)\Big(1-\frac{\left \| A \right \|_{F}^{2}}{\xi}\kappa(A)^{-2} \Big)\leq\Big[1-\frac{1}{2}\Big(\frac{1}{\gamma}\| A\|_{F}^{2}+1\Big)\kappa(A)^{-2}\Big]^{2},
\end{equation}
and Algorithm \ref{alg41} can converge faster than the greedy randomized
Kaczmarz method (GRK).
On the other hand, if $\theta=1$, we have from Theorem \ref{Thm2.3} and \eqref{311} that
\begin{equation}
\Big( 1-\frac{\left \| A \right \|_{F}^{2}}{\gamma }\kappa(A)^{-2} \Big)\Big(1-\frac{\left \| A \right \|_{F}^{2}}{\xi}\kappa(A)^{-2} \Big)\leq\Big[1-\frac{1}{\gamma}\| A\|_{F}^{2}\kappa(A)^{-2}\Big]^{2}.
\end{equation}
Thus, our bound \eqref{311} is sharper than \eqref{299} as $\theta=1$.
\end{remark}

\subsection{Random sampling for the partially randomized Kaczmarz method}

We have to scan all the rows of $A$ and calculate the probabilities corresponding to the residuals in GRK \cite{bibitem2} and RGRK \cite{bibitem3}. As a comparison, one has to seek the row with the largest (relative) residual in all the $m$ rows. Thus, it is time-consuming to determine the working rows when the size of the matrix is very large in the three algorithms.


In this section, we regard the selection of rows as random sampling, and do a simple random sampling before updating the approximate solution. The key is to use only a small portion of rows as samples, and then select working rows from the samples according to probabilities. Indeed, this idea stems from
Chebyshev's law of large numbers \cite{bibitem13}:
\begin{theorem}\label{Thm3.3} \cite{bibitem13}
Suppose that ${ z}_1,{ z}_2,\ldots,{ z}_n,\ldots$ is a series of independent random variables. They have expectation $\mathbb{E}({ z}_k)$ and variance respectively $\mathbb{D}({ z}_k)$. If there is s constant $C$ such that $\mathbb{D}({\bf x}_k)\leq C$, for any small positive number $\varepsilon$, we have
\begin{equation}
\lim_{n \to\infty  }P=\left \{ \left | \frac{1}{n} \sum_{k=1}^{n}{ z}_k-\frac{1}{n} \sum_{k=1}^{n}\mathbb{E}({ z}_k)\right |< \epsilon,\quad\forall\varepsilon>0  \right  \}=1.
\end{equation}
\end{theorem}

Theorem \ref{Thm3.3} indicates that if the sample size is large enough, the sample mean will approach to the population mean. Thus, the idea is to take a few rows into account, with no need to use all the rows.
On the other hand, according to Bernoulli's law of large numbers \cite{bibitem13}, when the number of experiments is large enough, the frequency of selecting each row is stable to its corresponding probability. Hence, the scheme of estimating the whole with the part is reasonable.

Furthermore, in order to avoid unreasonable sampling caused by randomness of simple sampling, we use ``$Z$ test" \cite{bibitem13} to evaluate the results of each random sampling. More precisely, given $0<\eta\ll 1$ and let $\{\omega_{1},\omega_{2},\ldots,\omega_{\eta m}\}$ be $\eta m$ simple random samples from the population with normal distribution $N\left(\mu,\sigma^2\right)$.
We assume that $\eta m$ is not too small, such that the sampling satisfies the Bernoulli's law of large numbers. For instance, we can set $\eta$ to be some empirical values such as 5\% \cite{bibitem13}.

The significant difference between the samples and the population can be judged by comparing the ``sample" Z-score
\begin{equation}\label{eq320}
Z=\frac{\overline{\bf \omega}-\mu }{s/\sqrt{\eta m}}
\end{equation}
with a ``theoretical" Z-score $q$ selected under the current distribution, where $\mu$ is the population mean, $\overline{\bf \omega}$ is the sample mean, and $s$ is the sample standard deviation. The choice of $q$ depends on the distribution of samples. For instance, if we choose $q=1.96$ under the normal distribution, the occurrence probability of significant difference will be no more than $5\%$ \cite{bibitem13}.
By \eqref{eq320}, if we set
$$
\mu=\frac{\sum_{i=1}^{m}\left \| A_{(i)} \right \|^{2}_{2}}{m},\quad \overline{\bf \omega}=\frac{\sum_{i=1}^{\eta m}\left \| \omega_{(i)} \right \|^{2}_{2}}{\eta m},\quad{\rm and}\quad s=\sqrt{\frac{\sum_{i=1}^{\eta m}\left(\left \| \omega_{(i)} \right \|^{2}_{2}-\overline{\bf \omega}\right)^2}{\eta m}},
$$
then the sample Z-score can be easily computed. If $Z<q$, we can accept the sampling, otherwise, we have to resample the population.
Similar to \eqref{399}, let the selected set be ${\Omega_k}$, in the proposed method, the probability of choosing the $ik$-th row is 1:
\begin{equation}\label{40}
\frac{|{\bf r}_{(ik)}|}{\| A_{(ik)}\|_{2}}=\frac{\left | {{\bf{b}}_{\left ( ik \right )}-A_{\left ( ik \right )}{\bf x}_{k}} \right |}{\left \| A_{\left ( ik \right )} \right \|_{2}}=\max\limits_{jk\in\Omega_k}\left \{ \frac{\left | {{\bf{b}}_{\left ( jk \right )}-A_{\left ( jk \right )}{\bf x}_{k}} \right |}{\left \| A_{\left ( jk \right )} \right \|_{2}}\right \}.
\end{equation}

We are ready to present the main algorithm of this paper. Notice tht we seek the working rows in a much smaller set $\Omega_k$ rather than all the $m$ rows, the new algorithm can not only reduce the workload in each step, but also save the storage requirements.
\begin{algorithm}\label{alg5}
{\bf A partially randomized Kaczmarz method with simple random sampling~(PRKS)}\\
{{\bf Input:} $A$, $\bf b$, ${\bf x}_0$, two parameters $\eta,q$, and the maximal iteration number $l$;}\\
{{\bf Output:} The approximate solution $\widetilde{\bf x}$;}\\
{\bf 1}. Compute the population mean $\mu=\frac{\sum_{i=1}^{m}\left \| A_{(i)} \right \|^2}{m}$;\\
{\bf 2}. for $k=0, 1, \ldots, l-1$ do
\begin{flushleft}
{\bf 3}. while $Z_k\geq q$ do\\
\quad \quad Randomly select $\eta m$ rows as samples, and calculate
\begin{equation*}
Z_k=\frac{\overline{\omega}_{k}-\mu }{s_{k}/\sqrt{\eta m}}
\end{equation*}
~~end\\
\end{flushleft}
{\bf 4}. Let the selected set be ${\Omega_k}$, and select $ik$ according to \eqref{40};
\\
{\bf 5}. Let ${\bf x}_{k+1}={\bf x}_k+\frac{{\bf b}_{\left ( ik \right )}-A_{\left ( ik \right )}{\bf x}_k}{\left \| A_{\left ( ik \right )} \right \|_{2}^{2}}\left ( A_{\left ( ik \right )} \right )^{H}$. If ${\bf x}_{k+1}$ is accurate enough, then stop, else continue; \\
{\bf 6}. endfor
\end{algorithm}



%
%

Next we give insight into the convergence of Algorithm \ref{alg5}.
Suppose that the row set of simple random sampling is $\Omega_k$, with number of samples being $\eta m$. Similar to \eqref{3.7}, we have that
\begin{align}
\mathbb{E}_{k}\left \|{\bf x}_{k+1}-{\bf x} \right \|_{2}^{2}&=\left \| {\bf x}_{k}-{\bf x} \right \|_{2}^{2}-\mathbb{E}\left \| {\bf x}_{k+1}-{\bf x}_{k} \right \|_{2}^{2}\nonumber\\
&=\left \| {\bf x}_{k}-{\bf x} \right \|_{2}^{2}-\max\limits_{jk\in\Omega_k}\left \{ \frac{\left | {{\bf{b}}_{\left ( jk \right )}-A_{\left ( jk \right )}{\bf x}_{k}} \right |^{2}}{\left \| A_{\left ( jk \right )} \right \|_{2}^{2}}\right \}\nonumber\\
&\leq\left \| {\bf x}_{k}-{\bf x} \right \|_{2}^{2}-\frac{\sum\limits_{jk\in\Omega_k}\left| {\bf r}_{jk} \right|^{2}}{\sum\limits_{jk\in\Omega_k}{\left\|A_{\left(jk\right)}\right\|_2^2}} \nonumber\\
&=\left \| {\bf x}_{k}-{\bf x} \right \|_{2}^{2}-\frac{\frac{1}{\eta m}\sum\limits_{jk\in\Omega_k}\left| {\bf r}_{jk} \right|^{2}}{\frac{1}{\eta m}\sum\limits_{jk\in\Omega_k}{\left\|A_{\left(jk\right)}\right\|_2^2}}.\label{3.30}
\end{align}
On the other hand, we have from \eqref{3.8} that
\begin{align}
\max_{1\leq i\leq m}\left \{ \frac{\left | {{\bf{b}}_{\left ( jk \right )}-A_{\left ( jk \right )}{\bf x}_{k}} \right |^{2}}{\left \| A_{\left ( jk \right )} \right \|_{2}^{2}}\right \}&=\frac{\max\limits_{1\leq i\leq m}\left \{ \frac{\left | {{\bf{b}}_{\left ( jk \right )}-A_{\left ( jk \right )}{\bf x}_{k}} \right |^{2}}{\left \| A_{\left ( jk \right )} \right \|_{2}^{2}}\right \}}{\left \| {\bf r}_{k} \right \|_{2}^{2}}\left \| {\bf r}_{k} \right \|_{2}^{2}\nonumber\\
&\geq \frac{\left \| {\bf r}_{k} \right \|_{2}^{2}}{\sum_{jk=1, jk\neq jk-1}^{m}\left \| A_{jk} \right \|_{2}^{2}}\nonumber\\
&=\frac{\frac{1}{m-1}\left \| {\bf r}_{k} \right \|_{2}^{2}}{\frac{1}{m-1}\left(\sum_{jk=1, jk\neq jk-1}^{m}\left \| A_{jk} \right \|_{2}^{2}\right)}.\label{3.44}
\end{align}

According to Chebyshev's law of large numbers \cite{bibitem13}, when $m$ is large enough and $\eta m$ is sufficiently large, there is a scalr $0<\varepsilon_k\ll 1$, such that \footnote{For instance, according to general statistical experience, the sampling error with respect to $\eta m=1200$ samples is about $\pm3\%$: https://baike.baidu.com/item/confidence level.}
\begin{equation}
\frac{\sum\limits_{jk\in\Omega_{k}}\left| {\bf r}_{jk} \right|^{2}}{\eta m}=\frac{\left \| {\bf r}_{k} \right \|_{2}^{2}}{m-1}\left(1\pm{\varepsilon_k}\right)\label{3.41},
\end{equation}
Similarly, in terms of Chebyshev’s law of large numbers \cite{bibitem13}, when $m$ is large enough and $\eta m$ is sufficiently large, there is a scalr $0<\widetilde{\varepsilon}_k\ll 1$, such that
\begin{equation}
\frac{\sum\limits_{jk\in\Omega_k}{\left\|A_{\left(jk\right)}\right\|_2^2}}{\eta m}=\frac{\sum\limits_{jk=1, jk\neq jk-1}^{m}\left \| A_{jk} \right \|_{2}^{2}}{m-1}\left(1\pm{\widetilde{\varepsilon}_k}\right)\label{3.43}.
\end{equation}
Let $\gamma$ and $\xi$ be defined in \eqref{eqq312}.
Combining \eqref {3.41}, \eqref{3.43} and \eqref{3.30}, we arrive at
\begin{align}
\mathbb{E}_{k}\left \|{\bf x}_{k+1}-{\bf x} \right \|_{2}^{2}&=\left \| {\bf x}_{k}-{\bf x} \right \|_{2}^{2}-\mathbb{E}\left \| {\bf x}_{k+1}-{\bf x}_{k} \right \|_{2}^{2}\nonumber\\
&\leq\left \| {\bf x}_{k}-{\bf x} \right \|_{2}^{2}-\frac{\frac{1}{\eta m}\sum\limits_{jk\in\Omega_k}\left| {\bf r}_{jk} \right|^{2}}{\frac{1}{\eta m}\sum\limits_{jk\in\Omega_k}{\left\|A_{\left(jk\right)}\right\|_2^2}}
\nonumber\\
&=\left\| {\bf x}_{k}-{\bf x} \right\|_{2}^{2}-\frac{\frac{1\pm\varepsilon_k}{m-1}\left\| {\bf r}_{k} \right\|_{2}^{2}}{\frac{1\pm\widetilde{\varepsilon}_{k}}{m-1}\left(\sum_{jk=1, jk\neq jk-1}^{m}\left\|A_{jk} \right\|_{2}^{2}\right)}\nonumber\\
&\leq\left\| {\bf x}_{k}-{\bf x} \right\|_{2}^{2}-\frac{\left(1-\varepsilon_k\right)\lambda_{\min}\left (A^{H}A\right )}{\left(1+\widetilde{\varepsilon}_k\right)\gamma }\left \| {\bf x}_{k}-{{\bf x}} \right \|_{2}^{2}\nonumber\\
&= \left ( 1-\frac{\left(1-{\varepsilon}_{k}\right)\left \| A \right \|_{F}^{2}}{\left(1+\widetilde{\varepsilon}_{k}\right)\gamma }\kappa \left ( A \right )^{-2}\right )  \left \| {\bf x}_{k}-{\bf x} \right \|_{2}^{2}.\label{3.45}
\end{align}

Similar to \eqref{3.11}, there exist $0\leq \varepsilon_{k-1},\widetilde{\varepsilon}_{k-1}\ll 1$, such that
\begin{align}
\mathbb{E}_{k-1}\left \|{\bf x}_{k}-{\bf x} \right \|_{2}^{2}&=\left \| {\bf x}_{k-1}-{\bf x} \right \|_{2}^{2}-\frac{\frac{1\pm\varepsilon_{k-1}}{m-1}\left({\bf r}_{(ik-1)}^2\right)}{\frac{1\pm\widetilde{\varepsilon}_{k-1}}{m-1}\left(\sum_{jk-1=1, jk-1\neq jk-2}^{m}\big(\left\| A \right\|_{F}^{2}-\left\| A_{(jk-1)} \right\|_{2}^{2}\big)\right)}\nonumber\\
&\leq \left \| {\bf x}_{k-1}-{\bf x} \right \|_{2}^{2}-\frac{\left(1-\varepsilon_{k-1}\right)\lambda_{\min}\left (A^{H}A\right )}{\left(1+\widetilde{\varepsilon}_{k-1}\right)\xi }\left \| {\bf x}_{k-1}-{{\bf x}} \right \|_{2}^{2}\nonumber\\
&= \left ( 1-\frac{\left(1-{\varepsilon}_{k-1}\right)\left \| A \right \|_{F}^{2}}{\left(1+\widetilde{\varepsilon}_{k-1}\right)\xi }\kappa \left ( A \right )^{-2}\right )  \left \| {\bf x}_{k-1}-{\bf x} \right \|_{2}^{2}.\label{3.46}
\end{align}

From \eqref{3.45} and \eqref{3.46}, we obtain the following theorem on the convergence of Algorithm \ref{alg5}.
\begin{theorem}\label{Thm3.4}
Under the above assumptions and notations, Algorithm \ref{alg5} converges to ${\bf x}$ in expectation, with
{\small\begin{eqnarray*}\label{3.23}
\mathbb{E}_{k}\left\| {\bf x}_{k+1}-{\bf x} \right\|_{2}^{2}\leq \left[\left ( 1-\frac{\left(1-{\varepsilon}_{k}\right)\left \| A \right \|_{F}^{2}}{\left(1+\widetilde{\varepsilon}_{k}\right)\gamma }\kappa \left ( A \right )^{-2}\right ) \left ( 1-\frac{\left(1-{\varepsilon}_{k-1}\right)\left \| A \right \|_{F}^{2}}{\left(1+\widetilde{\varepsilon}_{k-1}\right)\xi }\kappa \left ( A \right )^{-2}\right ) \right]\left\| {\bf x}_{k-1}-{\bf x} \right\|_{2}^{2}.
\end{eqnarray*}
}
\end{theorem}

It is seen from Theorem \ref{Thm3.4} and Theorem \ref{Thm3.2} that
the convergence speed of Algorithm \ref{alg5} can be slightly slow than Algorithm \ref{alg41}, i.e., the former may need more iterations than the latter. This is due to the fact that only a few rows of $A$ are utilized to choose the working rows.
In Figure 3.3, we plot the convergence curves of Algorithm \ref{alg5} with $\eta=0.5,0.2,0.05$, and those of Algorithm \ref{alg41} (i.e., Algorithm \ref{alg5} with $\eta=1$) and GRK. Here the coefficient matrix is randomly generated by using the MATLAB function $A=randn(200000,50)$.
We observe that the smaller $\eta$ is, the more iterations Algorithm \ref{alg5} uses. However, all of them converge faster than GRK.
On the other hand, the overhead in each step of Algorithm \ref{alg5} is much less than that of Algorithm \ref{alg41}.
As a result, Algorithm \ref{alg5} can run much faster than Algorithm \ref{alg41}. One refers to see Section \ref{sec5} for a comparison of Algorithm \ref{alg41} and Algorithm \ref{alg5}.


%
%


\begin{figure}[H]\label{Fig33}
\centering
\includegraphics[width =100 mm, height =60 mm]{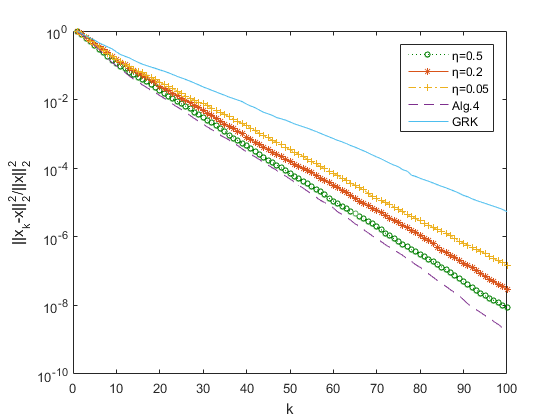}
\caption{Convergence curves of Algorithm \ref{alg5} with $\eta=0.5,0.2,0.05$, and those of Algorithm \ref{alg41} and GRK. The coefficient matrix is randomly generated by using the MATLAB function $A=randn(200000,50)$.}
\label{fig:3}
\end{figure}



\section{A Partially Kaczmarz Method with Simple Random Sampling for Ridge Regression}
\setcounter{equation}{0}

In this section, we are interested in the following ridge regression or the least squares regression problem
\begin{equation}\label{41}
\underset{\bf x}{\min}\left \| A\bf{x}-\bf{b} \right \|_2+\tau \left \| \bf{x} \right \|_2,
\end{equation}
where $A$ is an $m\times n$ (complex) matrix, ${\bf b}$ is a given $m$-dimensional (complex) vector, $\tau>0$ is a given positive parameter, and ${\bf x}$ is the desired solution vector.
This type of problem arises from many practical problems such as machine learning \cite{Zhang}, statistical analysis \cite{Haw}, ill-posed problem \cite{Tik,Tik2}, and so on \cite{Had,bibitem16}.

By taking derivative with respect to ${\bf x}$, the above optimization problem can be computed via solving the following linear systems
\begin{equation}\label{42}
\left (K +\tau I_{m}  \right )y={\bf{b}},\quad K =A A^{H},
\end{equation}
or
\begin{equation}\label{411}
\left (\Sigma +\tau I_{n}  \right )x=A^{H}{\bf{b}},\quad {\Sigma} =A^{H}A.
\end{equation}

In \cite{Hef}, Hefny {\it et al.} give a variant of randomized Kaczmarz (VRK) algorithm and a variant of randomized Gauss-Seidel (VRGS) algorithm to solve \eqref{42} and \eqref{411}, respectively.
In \cite{bibitem16}, Ivanov and Zhdanov solve the problem of \eqref{41} by applying the randomized Kaczmarz algorithm to a regularized normal equation of the form \eqref{411}.
Recently, Gu and Liu \cite{bibitem14} extend the GRK algorithm to solve \eqref{42}, and present a variant of greedy randomized Kaczmarz algorithm with relaxation parameter \big(VGRKRP($\omega$)\big). In addition, an accelerated GRK algorithm with relaxation for ridge regression \big(AVGRKRP($\omega$)\big) is proposed, by executing more rows that corresponding to the larger entries of the residual vector simultaneously at each iteration.

However, in all the algorithms proposed in \cite{bibitem14}, one has to explicitly compute and store the $m$-by-$m$ matrix $K$, which is unfavorable or even infeasible for large scale problems. Moreover,
the {\it optimal} relaxation parameters required in VGRKRP($\omega$) and AVGRKRP($\omega$) are difficult to determine in advance. In this section, we solve the ridge regression problem \eqref{41} by applying our proposed algorithms to \eqref{42}, with no need to form the matrix $K$ explicitly, and our algorithms are free of choosing relaxation parameters.
In terms of \eqref{39}, we select the working row number $ik$ such that 
\begin{equation}\label{44}
\frac{\Big|{\bf{b}}_{\left ( ik \right )}-\big(A_{(ik)}\cdot A^H+\tau I_{(ik)}\big){\bf x}_k\Big|}{\|A_{(ik)}\cdot A^H+\tau I_{(ik)}\|_{2}}=\max\limits_{1\leq jk\leq m}\left \{ \frac{\Big|{\bf{b}}_{\left ( jk \right )}-\big(A_{(jk)}\cdot A^H+\tau I_{(jk)}\big){\bf x}_k\Big|}{\|A_{(jk)}\cdot A^H+\tau I_{(jk)}\|_{2}}\right \}.
\end{equation}


Based on Algorithm \ref{alg41} and the above discussions, we have the following algorithm.
\begin{algorithm}\label{alg9}
{\bf A partially randomized Kaczmarz method for ridge regression problems}\\
{{\bf Input:} $A$, $\bf b$, $l$ and $\bf{x_0}$, where $l$ is the maximal iteration number;}\\
{{\bf Output:}  The approximate solution $\widetilde{\bf x}$;}\\
{\bf 1}.  for $k=0, 1, \ldots, l-1$ do\\
{\bf 2}. Select the working row number $ik$ as in \eqref{44}; 
\\
{\bf 3}. Let ${\bf x}_{k+1}={\bf x}_k+\frac{{\bf{b}}_{\left ( ik \right )}-\big(A_{(ik)}\cdot A^H+\tau I_{(ik)}\big){\bf x}_k}{{\bf y}_{(ik)}^{2}}\big(A_{(ik)}\cdot A^H+\tau I_{(ik)}\big)^{H}$. If ${\bf x}_{k+1}$ is accurate enough, then stop, else continue;\\
{\bf 4.} endfor
\end{algorithm}

In each iteration, the main overhead is to compute two matrix-vector products with respect to $A$ and $A^H$, and there is no need to form and store the $m$-by-$m$ matrix $AA^{H}$. So Algorithm \ref{alg9} is much cheaper and requires fewer storage than the algorithms presented in \cite{bibitem14}.
Similar to Theorem \ref{Thm3.2}, we have the following algorithm for the convergence of Algorithm \ref{alg9}.

%
%
%
\begin{theorem}
Let ${\bf x}$ be the solution of \eqref{42} and let $\left( K+\tau I_{m} \right)_{\left(s\right)}$ be the second smallest row of $K+\tau I_{m}$ in 2-norm. Then Algorithm \ref{alg9} converges to ${\bf x}$ in expectation, with
{\small
$$
\mathbb{E}_{k}\left \| {\bf x}_{k+1}-{\bf x} \right \|_{2}^{2}\leq \left[\left( 1-\frac{\left \| \left( K+\tau I_{m} \right) \right \|_{F}^{2}}{\nu}\kappa^{-2} \left ( K+\tau I_{m} \right ) \right)\left( 1-\frac{\left \| \left ( K+\tau I_{m} \right ) \right \|_{F}^{2}}{\rho}\kappa^{-2} \left ( K+\tau I_{m} \right ) \right)\right]\left \| {\bf x}_{k-1}-{\bf x} \right \|_{2}^{2},
$$
}
where
\begin{equation}\label{445}
\nu =\underset{1\leq i\leq m}{\max}\sum_{j=1,j\neq i}^{m}\|\left( K+\tau I_{m} \right)_{\left(j\right)}\|_2,~~{\rm and}~~\rho=\left\|  \left( K+\tau I_{m} \right) \right\|_{F}^{2}-\left\|  \left( K+\tau I_{m} \right)_{\left(s\right) }\right\|_{2}^{2}.
\end{equation}
\end{theorem}

\begin{remark}
We have to calculate the 2-norms of all the rows $\|A_{(jk)}\cdot A^H+\tau I_{(jk)}\|_{2},jk=1,2,\ldots,m$, which requires $m$ matrix-vector products and thus will be time-consuming. In order to save computational overhead,
we consider the two nonnegative vectors as follows
$$
{\bf y}_{1}={\rm abs}\left (( {\bf{e}}^{H}A)A^{H})+\tau{\bf e}^{H} \right)\quad {\rm and}\quad {\bf y}_{2}=\left (( {\bf{e}}^{H}{\rm abs}\left ( A \right )){\rm abs}( A^{H})+\tau{\bf e}^{H}\right ),
$$
where ${\bf e}\in\mathbb{R}^{m}$ is the vector of all ones. Let ${\bf z}\in\mathbb{R}^m$ be the vector composed of the 2-norms of the $m$ rows of $AA^{H}+\tau I$, then both ${\bf z}$ and ${\bf y}_1$ are {\it elementwise} smaller than ${\bf y}_2$. The idea is to make use of
\begin{equation}\label{eq45}
{\bf y}=\frac{{\bf y}_{1}+{\bf y}_{2}}{2}
\end{equation}
to approximate ${\bf z}$, i.e., using ${\bf y}_{(ik)}$ as an estimation to $\|A_{(ik)}\cdot A^H+\tau I_{(ik)}\|_{2}$. Note that the vector ${\bf y}$ needs to compute only once and store for a latter use. More precisely, we exploit
\begin{equation}\label{47}
\frac{\Big|{\bf{b}}_{\left ( ik \right )}-\big(A_{(ik)}\cdot A^H+\tau I_{(ik)}\big){\bf x}_k\Big|}{{\bf y}_{(ik)}}=\max\limits_{1\leq jk\leq m}\left \{ \frac{\Big|{\bf{b}}_{\left ( jk \right )}-\big(A_{(jk)}\cdot A^H+\tau I_{(jk)}\big){\bf x}_k\Big|}{{\bf y}_{(jk)}}\right \}.
\end{equation}
to take the place of \eqref{44} in Step 4 of Algorithm \ref{alg9}.
Numerical experiments show that the estimation is very effective.
\end{remark}


However, one has to seek the working row corresponding to the largest (relative) residual element in magnitude among all the $m$ rows.
Similar to Algorithm \ref{alg5}, to release the overhead, we use only a small portion of rows as samples, and choose working rows by using simple random
sampling. This gives the following algorithm:


\begin{algorithm}\label{alg10}
{\bf A partially randomized Kaczmarz method with simple random
sampling for ridge regression}\\
{{\bf Input:} $A$, $\bf b$, $\tau$, $l$ and $\bf{x_0}$, where $l$ is the maximal iteration number;}\\
{{\bf Output:}  The approximate solution $\widetilde{\bf x}$;}\\
{\bf 1}. for $k=0, 1, \ldots, l-1$ do\\
{\bf 2}. Choose $\eta m$ rows as samples by using simple random
sampling;\\
{\bf 3}. Let the selected set be ${\Omega_k}$, and choose the working row number $ik$, such that
\begin{equation}\label{eq48}
\frac{\Big|{\bf{b}}_{\left ( ik \right )}-\big(A_{(ik)}\cdot A^H+\tau I_{(ik)}\big){\bf x}_k\Big|}{\|A_{(ik)}\cdot A^H+\tau I_{(ik)}\|_{2}}=\max\limits_{jk\in\Omega_k}\left \{ \frac{\Big|{\bf{b}}_{\left ( jk \right )}-\big(A_{(jk)}\cdot A^H+\tau I_{(jk)}\big){\bf x}_k\Big|}{\|A_{(jk)}\cdot A^H+\tau I_{(jk)}\|_{2}}\right \}.
\end{equation}
\\
{\bf 4}.  Let ${\bf x}_{k+1}={\bf x}_k+\frac{{\bf{b}}_{\left ( ik \right )}-\big(A_{(ik)}\cdot A^H+\tau I_{(ik)}\big){\bf x}_k}{\|A_{(ik)}\cdot A^H+\tau I_{(ik)}\|_{2}^{2}}\big(A_{(ik)}\cdot A^H+\tau I_{(ik)}\big)^{H}$. If ${\bf x}_{k+1}$ is accurate enough, then stop, else goto Step 2;\\
{\bf 5.} endfor
\end{algorithm}

The following theorem shows the convergence of Algorithm \ref{alg10}, whose proof is similar to that of Theorem \ref{Thm3.4} and thus is omitted:
{\begin{theorem}
Let ${\bf x}$ be the solution of \eqref{42}, then there exist $0<\widehat{\varepsilon}_k,\breve{\varepsilon}_k,\widehat{\varepsilon}_{k-1},\breve{\varepsilon}_{k-1}\ll 1$, such that 
{\small
$$
\mathbb{E}_{k}\left \| {\bf x}_{k+1}-{\bf x} \right \|_{2}^{2}\leq \chi \left\| {\bf x}_{k-1}-{\bf x} \right \|_{2}^{2},
$$
}
where
{\small
$$
\chi=\left[\left(1-\frac{\left({1-\widehat{\varepsilon}_k}\right)\left \| \left ( K+\tau I_{m} \right ) \right \|_{F}^{2}}{\left({1+\breve{\varepsilon}_k}\right)\nu}\kappa^{-2}\left ( K+\tau I_{m} \right ) \right)\left( 1-\frac{\left({1-\widehat{\varepsilon}_{k-1}}\right)\left \| \left ( K+\tau I_{m} \right ) \right \|_{F}^{2}}{\left({1+\breve{\varepsilon}_{k-1}}\right)\rho}\kappa^{-2}\left ( K+\tau I_{m} \right ) \right)\right],
$$
}
and $\nu,\rho$ are defined in \eqref{445}.
\end{theorem}}

However, we have to the 2-norms of all the rows of $K+\tau I$ in Algorithm \ref{alg10}. Similarly, one can compute the vector ${\bf y}$ as in \eqref{eq45}, and replace \eqref{eq48} by
\begin{equation}\label{49}
\frac{\Big|{\bf{b}}_{\left ( ik \right )}-\big(A_{(ik)}\cdot A^H+\tau I_{(ik)}\big){\bf x}_k\Big|}{{\bf y}_{(ik)}}=\max\limits_{jk\in\Omega_k}\left \{ \frac{\Big|{\bf{b}}_{\left ( jk \right )}-\big(A_{(jk)}\cdot A^H+\tau I_{(jk)}\big){\bf x}_k\Big|}{{\bf y}_{(jk)}}\right \}.
\end{equation}
in Step 3 of Algorithm \ref{alg10}.

\section{Numerical Experiments}\label{sec5}

In this section, we perform some numerical experiments to show the numerical behavior of our proposed algorithms. All the numerical experiments are obtained from using MATLAB 2018b, on a Hp workstation with 20 cores double Intel(R)Xeon(R) E5-2640 v3 processors, with CPU 2.60 GHz and RAM 256 GB. The operation system is 64-bit Windows 10.

In order to show the efficiency of our proposed algorithms for solving the large-scale linear system \eqref{1.1}, we compare our proposed algorithms Algorithm \ref{alg41} and Algorithm \ref{alg5} with some recently proposed Kaczmarz  algorithms including:\\
$\bullet$ {\bf RK}: The randomized Kaczmarz method proposed in \cite{bibitem8}.\\
$\bullet$ {\bf GRK}: The greedy randomized Kaczmarz method due to Bai and Wu \cite{bibitem3}.\\
$\bullet$ {\bf RGRK}: The relaxed greedy randomized Kaczmarz method \cite{bibitem2}.\\

For the ridge regression problem \eqref{41}, we compare the proposed algorithm Algorithm \ref{alg9} and Algorithm \ref{alg10} with \footnote{ We thank Dr. Yong Liu for providing us MATLAB files of VGRK and AVGRKRP~($\omega$).}:\\
$\bullet$ {\bf VRK}: A variant of randomized Kaczmarz algorithm \cite{Hef}.\\
$\bullet$ {\bf VRGS}: A randomized Gauss-Seidel algorithm \cite{Hef}.\\
$\bullet$ {\bf VGRK}: The variant of greedy randomized Kaczmarz algorithm \cite{bibitem14}. \\
$\bullet$ {\bf AVGRKRP~($\omega$)} : The accelerated variant of greedy randomized Kaczmarz algorithm with relaxation parameter \cite{bibitem14}.\\

As was done in \cite{bibitem1}, we make use of the vector ${\bf x}=[1,1,\ldots ,1]^T$ as the ``exact" solution and set the right-hand-side ${\bf b}=\mathcal{A}{\bf x}$, where $\mathcal{A}=A$ for the large-scale linear system \eqref{1.1} and $AA^H+\tau I_m$ for the ridge regression problem \eqref{42}. The initial vector ${\bf x}_0$ is chosen as the zero vector. The stopping criterion is
$$
err=\frac{\left \| {\bf x}-{\bf x}_k\right \|_{2}^{2}}{\left \| {\bf x}_k \right \|_{2}^{2}}<tol,
$$
where ${\bf x}_k$ is the approximation from the $k$-th iteration and $tol$ is a user-described tolerance. If the number of iteration exceeds 400,000, or the CPU time exceeds 12 hours, we will stop the algorithm and declare it fails to converge.
In all the tables below, we denote by ``IT" the number of iterations, and by ``CPU" the CPU time in seconds. All the experiments are repeated for 5 times, and the iteration numbers as well as the CPU time in seconds, are the mean from the 5 runs.

\begin{table}[H]
\centering
\caption{Section 5.1: Test of the linear systems on the coefficient matrix $A\in \mathbb{R}^{m\times n}$ generated by using the MATLAB build-in function {\tt randn(m,n)}, $tol=1e-6$. The sampling ratio is chosen as $\eta=0.05$ in Algorithm \ref{alg5}.}
{\footnotesize\begin{tabular}{|c|c|c|c|c|c|c|}
\hline
\multicolumn{2}{|c|}{$m\times n$}   & {$1000\times200$} & {$2000\times400$} & {$3000\times600$} & {$4000\times800$} & {$5000\times1000$} \\ \hline
\multirow{2}{*}{RK}   & IT  & 3790     & 7247     & 11488     & 15026     & 19115     \\ \cline{2-7}
                      & CPU & 0.41   & 1.07     & 1.94   & 3.03   & 4.82    \\ \hline
\multirow{2}{*}{GRK}  & IT  & 593     & 1182      & 1694      & 2267      & 2817     \\ \cline{2-7}
                      & CPU & 0.14     & 0.48    & 0.83    & 2.71    & 5.63   \\ \hline
\multirow{2}{*}{RGRK $\left(\theta=0.75\right)$}  & IT  & 542     & 1147      & 1605      & 2123      & 2767     \\ \cline{2-7}
                      & CPU & 0.13     & 0.44    & 0.81    & 2.56    & 5.42   \\ \hline
\multirow{2}{*}{RGRK $\left(\theta=1\right)$}  & IT  & 511     & 1122     & 1509     & 1994      & 2724      \\ \cline{2-7}
                      & CPU & 0.11     & 0.43    & 0.77    & 2.51   & 5.37   \\ \hline
\multirow{2}{*}{\bf Algorithm 4} & IT  & 511     & 1122     & 1509      & 1994      & 2724      \\ \cline{2-7}
                      & CPU & {\bf0.06}   & {\bf0.21}    &  {\bf0.45}  & 1.76    &  4.16   \\ \hline
\multirow{2}{*}{\bf Algorithm 5} & IT  & 676    & 1253      & 1840     & 2596      & 3248      \\ \cline{2-7}
                      & CPU & 0.15   & 0.26    & {0.51}    &  {\bf0.71}   & {\bf 2.36}   \\ \hline
\end{tabular}}\label{T533}
\end{table}

\begin{table}[h]
\centering
\caption{Section 5.1: Test of the linear systems on the coefficient matrix $A$ generated by using the MATLAB build-in function {\tt randn(m,n)}, $tol=1e-6$. The sampling ratio is chosen as $\eta=0.001$ in Algorithm \ref{alg5}.
}
{\footnotesize\begin{tabular}{|c|c|c|c|c|c|c|}
\hline
\multicolumn{2}{|c|}{$m\times n$}   & {$300000\times50$} & {$300000\times100$} & {$300000\times500$} & {$300000\times1000$} & {$300000\times5000$} \\ \hline
\multirow{2}{*}{RK}   & IT  & 697     & 1420     & 6881     & 13790    & 69931     \\ \cline{2-7}
                      & CPU & 8.42    & 15.72     & 74.88   & 149.70    & 891.70    \\ \hline
\multirow{2}{*}{GRK}  & IT  & 63     & 174      & 874      & 1007      & 5620     \\ \cline{2-7}
                      & CPU & 1.68     & 2.16    & 25.61    & 84.99    & 2600.35   \\ \hline
\multirow{2}{*}{RGRK $\left(\theta=0.75\right)$}  & IT  & 49     & 152      & 759     & 835      & 5198     \\ \cline{2-7}
                      & CPU & 1.00     & 2.77    & 21.25    & 69.47    & 2250.47   \\ \hline
\multirow{2}{*}{RGRK $\left(\theta=1\right)$}  & IT  & 32    & 116     & 748      & 638      & 4901     \\ \cline{2-7}
                      & CPU & 0.83     & 2.10    & 20.11    & 61.34    & 2044.13   \\ \hline
\multirow{2}{*}{\bf Algorithm 4} & IT  & 32    & 116     & 748      & 638      & 4901     \\ \cline{2-7}
                      & CPU & 0.31   & 0.78    & 16.63   & 50.70    & 2014.85    \\ \hline
\multirow{2}{*}{\bf Algorithm 5} & IT  & 51     & 178      & 823     & 1411      & 7396      \\ \cline{2-7}
                      & CPU & \bf{0.25}   & \bf{0.23}    & \bf{1.97}    & \bf{11.48}    & \bf{394.7}   \\ \hline
\end{tabular}}\label{T54}
\end{table}

\subsection{Numerical Experiments on Large Linear Systems with Synthetic Data}

In this subsection, we use some synthetic data generated randomly by using the MATLAB build-in function {\tt randn} for the large linear systems \eqref{1.1}. In the first example, the test matrices are generated by
{\tt A=randn(5n,n)}, with $n=200,400,600,800$ and 1000, respectively; see Table \ref{T533}. In the second example, the test matrices are generated by
{\tt A=randn(300000,n)}, with $n=50,100,500,1000$ and 5000, respectively; see Table \ref{T54}.
In this subsection, we run the algorithms RK, GRK, RGRK, Algorithm \ref{alg4}, Algorithm \ref{alg41}, and Algorithm \ref{alg5}, with convergence tolerance $tol=10^{-6}$.
Specifically, we run the relaxed greedy randomized Kaczmarz method (RGRK) with both $\theta=0.75$ (the one used in \cite{bibitem2}) and $\theta=1$ (the theoretically optimal parameter). Tables \ref{T533}--\ref{T54} list the numerical results.

From Tables \ref{T533}--\ref{T54}, it is seen that both GRK and RGRK use much fewer iterations than the RK method, and require less CPU time than the RK method in the most cases, except for the number of columns is large, say, $n=5000$. This is because one has to to determine some index sets such as \eqref{2.5} in GRK and RGRK, whose workload is large, especially for big data problems. This is also the reason why Algorithm \ref{alg41} outperforms the relaxed greedy randomized Kaczmarz method  with the relaxation parameter $\theta=1$, even if they are mathematically equivalent; see Remark 3.1.
As a comparison, our three new algorithms use comparable iteration numbers to GRK and RGRK, and Algorithm \ref{alg41} converges faster than RK, GRK, RGRK, and Algorithm \ref{alg4}, while Algorithm \ref{alg5} performs the best in terms of CPU time.
Notice that there is no need to calculate probabilities in the three proposed algorithms.

Indeed, the cost of random selection according to probability will be high when the matrix in question is very large.
Fortunately, in Algorithm \ref{alg5}, we only use a few rows of $A$ for the working rows, and thus the computational overhead
in each iteration of Algorithm \ref{alg5} can be much less than those of the others. So we benefit from this strategy significantly, which speeds up the calculation.

%

\subsection{Numerical Experiments on Large Linear Systems with Real Data}

In this section, we run our proposed algorithms on some real data. The large sparse data matrices are from the University of Florida Sparse
Matrix Collection \footnote{https://sparse.tamu.edu/}, and the large dense matrixes are from the YouTube Faces data set\footnote{http://www.cs.tau.ac.il/~wolf/ytfaces/}, the StarPlus fMRI data set\footnote{http://www.cs.cmu.edu/afs/cs.cmu.edu/project/theo-81/www/}, as well as the Flint data set \footnote{https://www.jianshu.com/p/5fde55a4d267?tdsourcetag=s\_pcqq\_aiomsg}.
The details of these data matrices are given in Table \ref{T51}, where the matrix $201912\_75N060W\_AVE\_5\times 1$ is obtained from extending the width of the original picture to five times via the MATLAB function {\tt imresize}, and the matrices $bibd\_17\_8^T$ and $mri2^T$ are transpose of matrices $bibd\_17\_8$ and $mri2$, respectively.
We run RK, GRK, RGRK, Algorithm \ref{alg4}, Algorithm \ref{alg41}, and Algorithm \ref{alg5} on these problems, and the convergence tolerance is chosen as $tol=10^{-3}$. In Table \ref{TT54}, we present the numerical results performed on sparse matrices from the University of Florida Sparse Matrix Collection, and the sampling ratio is chosen as $\eta=0.01$ in Algorithm \ref{alg5}. In Table \ref{TT55}, we list on the numerical results on the dense matrices from Luminous remote sensing, YouTube Faces data set, and StarPlus fMRI data set, where the sampling ratio is chosen as $\eta=0.001$ in Algorithm \ref{alg5}. Here ``/" implies the number of iterations exceeds 400000 or the CPU time exceeds 12 hours.

\begin{table}[H]
\caption{Test matrices used in Section 5.2 for solving large linear systems, where ``$^T$" denotes transpose of the matrix.}
\centering
{\footnotesize\begin{tabular}{|c|c|c|c|}
  \hline
  Matrix & size ($m\times n$) & nnz & Background \\\hline
  $abtaha1$ & {$14596\times209$}& 51307 & Combinatorial Problem \\
  $abtaha2$ & {$37932\times331$} & 137228 & Combinatorial Problem \\
  $bibd\_17\_8^T$ & {$24310\times136$} & 680680 & Combinatorial Problem \\
  $sls$ & {1748122$\times62729$} & 6804304 & Least Squares Problem \\
  $airfoil1\_dual$ & {$8034\times8034$} & 23626 & 2D/3D Problem \\
  $mri2^T $&$147456\times63240$ & 596160 & MRI Problem \\\hline
  $YouTubeFaceArrange\_64\times64 $& $370319\times4096$ & full & Face data \\
  $YouTubeFaceArrange\_128\times128 $& $370319\times16384$ & full & Face data \\
  $201912\_75N060W\_AVE\_5\times 1$ & $90000\times28800$ & full & Luminous remote sensing \\
  \hline
\end{tabular}}\label{T51}
\end{table}

Again, we see from Tables \ref{TT54}--\ref{TT55} that Algorithm \ref{alg41} and Algorithm \ref{alg5} outperform the other algorithms, and their CPU time is much less than those of the GRK method and the RGRK method.
Specifically, our new algorithms may converge faster than GRK and RGRK even if they share about the same iteration numbers. For instance, for the {\it airfoil1\_dual} matrix, we observe from Table \ref{TT54} that both RGRK $\left(\theta=1\right)$ and Algorithm \ref{alg41} use about the same iterations, while the latter is about two times faster than the former. Indeed, the RGRK method with $\theta=1$ is mathematically equivalent to Algorithm \ref{41}, while the proposed algorithms are much cheaper (per iteration) than GRK and RGRK.
More precisely, it was stressed that $\theta=1$ may not be a good choice \cite{bibitem2}. However, we see from the tables that RGRK with $\theta=1$ works better than GRK in most cases, and this is consistent with our analysis given in Section \ref{Sec3}. As there is no need to form the greedy index set, nor to calculate the probability for choosing working rows, Algorithm \ref{alg41} is superior to the RGRK method in terms of CPU time.


It is seen from Table \ref{TT55} that the RK, GRK, and RGRK methods do not work for many problems. This is because they cannot calculate the probabilities accurately during iterations, which suffer from rounding errors. Even if one can correct the probabilities so that the algorithms keep working, the additional overhead will be large. Notice that Algorithm \ref{alg5} requires only a small part of the matrix in each iteration, while the RK, GRK and RGRK methods have to scan all the rows for calculation. Consequently,  Algorithm \ref{alg5} is often superior to the others, and has more suitable to large-scale and dense problems.

\begin{table}[h]
\centering
\caption{Section 5.2: Test of the linear systems on the sparse matrices from the University of Florida Sparse Matrix Collection, $tol=1e-3$. The sampling ratio is chosen as $\eta=0.01$ in Algorithm \ref{alg5}. Here ``/" means the number of iterations exceeds 400000 or the CPU time exceeds 12 hours.}
{\footnotesize\begin{tabular}{|c|c|c|c|c|c|c|}
\hline
\multicolumn{2}{|c|}{Matrix\& Size}   & \makecell[c]{$abtaha1$ \\{$14596\times209$}} & \makecell[c]{$abtaha2$ \\{$37932\times331$}} & \makecell[c]{$sls$\\{1748122$\times62729$}} & \makecell[c]{$bibd\_17\_8^T$\\$24310\times136$}& \makecell[c]{$airfoil1\_dual$\\$8034\times8034$} \\ \hline
\multirow{2}{*}{RK}   & IT  & 35120     & 57756     & /     & 1449    & 113026     \\ \cline{2-7}
                      & CPU & 1204.33    & 4890.46     & /   & 82.53    & 2142.28   \\ \hline
\multirow{2}{*}{GRK}  & IT  & 486     & 617      & 74334\      & 141     &57901      \\ \cline{2-7}
                      & CPU & 0.83    & 2.96    & 13299.49    & 0.63    & 38.06   \\ \hline
\multirow{2}{*}{RGRK $\left(\theta=0.75\right)$}  & IT  & 486     & 617      & 74334\      & 141     &57901      \\ \cline{2-7}
                      & CPU & 0.83    & 2.96    & 13297.21    & 0.63    & 38.06   \\ \hline
\multirow{2}{*}{RGRK $\left(\theta=1\right)$} & IT  & 210    & 267      & 62737\      & 108     & 60316       \\ \cline{2-7}
                      & CPU & 0.36   & 1.37    & 12449.93    & 0.39    &40.63   \\ \hline
\multirow{2}{*}{\bf Algorithm 4} & IT  & 210    & 267      & 62737\      & 108     & 60316       \\ \cline{2-7}
                      & CPU & \bf{0.07}   & {\bf 0.21}    & 9727.34\    & {\bf 0.28}   &{22.79}     \\ \hline
\multirow{2}{*}{\bf Algorithm 5} & IT  & 645     & 677      & 80639      &192      &55943\      \\ \cline{2-7}
                      & CPU & 0.21   & { 0.48}    & {\bf 4391.92}   & {0.39}   & {\bf 22.15}\    \\ \hline
\end{tabular}}\label{TT54}
\end{table}

\begin{table}[h]
\centering
\caption{Section 5.2: Test of the linear systems on the dense matrices from Luminous remote sensing, YouTube Faces data set, and StarPlus fMRI data set, $tol=1e-3$. The sampling ratio is chosen as $\eta=0.001$ in Algorithm \ref{alg5}. Here ``/" means the number of iterations exceeds 400000 or the CPU time exceeds 12 hours.}
{\footnotesize\begin{tabular}{|c|c|c|c|c|c|}
\hline
\multicolumn{2}{|c|}{Matrix\& Size} & \makecell[c]{\scriptsize{$201912\_75N060W\_AVE\_5\times 1$}\\$90000\times28800$} &\makecell[c]{\scriptsize{$YouTubeFaceArrange\_64\times64$}\\$370319\times4096$}  &\makecell[c]{\scriptsize{$YouTubeFaceArrange\_128\times128$}\\$370319\times16384$}  &\makecell[c]{\scriptsize{$mri2^T$}\\$147456\times63240$}  \\
\hline
\multirow{2}{*}{RK}  &  IT  &/  &319191  &/  &/  \\ \cline{2-6}
 &CPU  &/  &3971.63  &/  &/  \\ \hline
\multirow{2}{*}{GRK}  &  IT  &/  &1009  &1497  &265586  \\ \cline{2-6}
 &CPU  &/  &504.91  &2109.61  &4673.75  \\ \hline
\multirow{2}{*}{RGRK $\left(\theta=0.75\right)$}   &IT   &/  &862  &1288  & 263310 \\ \cline{2-6}
 &CPU  &/  &477.25  &1793.57  &4660.49  \\ \hline
\multirow{2}{*}{RGRK $\left(\theta=1\right)$}   &IT   &/  &341  &552  &284843 \\ \cline{2-6}
 &CPU  &/  &225.79  &1142.20  &4932.33  \\ \hline
\multirow{2}{*}{\bf Algorithm 4}  &IT  &34944  &341  &552  &284843  \\ \cline{2-6}
 &CPU  &8793.66  &166.2  &756.19  &  1883.25\\ \hline
\multirow{2}{*}{\bf Algorithm 5}  &IT  &90157  &4836  &4976  &239295  \\ \cline{2-6}
 &CPU  &{\bf5795.28}  &{\bf124.31}  &{\bf415.70}  &{\bf1497.94}  \\ \hline
\end{tabular}}\label{TT55}
\end{table}

\begin{table}[H]
\caption{Test matrices used in Section 5.3 for ridge regression problems}
\centering
\begin{tabular}{|c|c|c|c|}
  \hline
  Matrix & Size ($m\times n$) & nnz & Background \\\hline
  $ch7-8-b2$ & {$11760\times1760$}& 35280 & Combinatorial Problem \\
  $ch7-9-b2$& {$17640\times1512$} & 52920 & Combinatorial Problem \\
  $ch8-8-b2$ & {$18816\times1568$} & 56448 & Combinatorial Problem \\
  $ch6-6-b2$ & {$2400\times450$}& 7200 & Combinatorial Problem \\
  $bcsstm09$ & {$1083\times1083$} & 1083 & Structural Problem \\
  \hline
\end{tabular}\label{T56}
\end{table}

\subsection{Numerical Experiments on Ridge Regression Problems}

In this example, we consider the ridge regression problem \eqref{42}. The test matrices are listed in Table \ref{T56}, which are from the University of Florida Sparse
Matrix Collection \footnote{https://sparse.tamu.edu/}. We run the
VRK method \cite{Hef}, the VRGS method \cite{Hef}, the VGRK method \cite{bibitem14}, the AVGRK~($\omega$) method \cite{bibitem14}, as well as Algorithm \ref{alg9} and Algorithm \ref{alg10} on this problem. Notice that both VGRK and AVGRK~($\omega$) need to for the matrix $K$ explicitly, refer to \eqref{42}, so the CPU time of these two algorithms include both that for forming $K$ and for solving \eqref{42} iteratively.

In all the algorithms, the convergence tolerance is chosen as  $tol=10^{-3}$, and the regularization parameters are chosen as $\tau=0.1,0.01$ and 0.001, respectively. The sampling ratio is chosen as $\eta=0.01$ in Algorithm \ref{alg10}. As was done in \cite{bibitem14}, we pick the relaxation parameter $\omega=1+\frac{n}{m}$ in AVGRK~($\omega$).
If the number of iterations of an algorithm exceeds 400000, or the CPU time is over 12 hours, we declare that the algorithm fails to converge. Tables \ref{T57}--\ref{T59} present the numerical results.


Some remarks are in order. First,
we observe from Tables \ref{T57}--\ref{T59} that, all the algorithms VRK, VRGS, VRGK and AVGRK~($\omega$) do not work for these problems in most cases. As a comparison, Algorithm \ref{alg9} and Algorithm \ref{alg10} run quite well. These show the superiority of our proposed algorithms over many state-of-the-art algorithms for ridge regression problems. Second, we see that AVGRK~($\omega$) perform better than VRK, VRGS and VRGK when $\tau=0.1$. However, AVGRK~($\omega$) uses much more iterations and CPU time than Algorithm \ref{alg9} and Algorithm \ref{alg10}.
For the {\tt bcsstm09} matrix, it is seen that AVGRK~($\omega$) only requires 8 iterations and Algorithm \ref{alg9} needs 25301 iterations, while the CPU time of the two algorithms are comparable. Indeed, AVGRK~($\omega$) tries to use all the information contained in the indicator set, so the number of iterations of AVGRK~($\omega$) can be small. However, the costs of AVGRK~($\omega$) in each iteration is much larger than those of Algorithm \ref{alg9} and Algorithm \ref{alg10}.
Third, unlike AVGRK~($\omega$), we see that Algorithm \ref{alg9} and Algorithm \ref{alg10} are insensitive to the regularization parameter $\tau$ used.
Recall that there is no need to form and store $K$ explicitly in our two proposed algorithms, moreover, they are free of relaxation parameters. Consequently, our new algorithms are competitive candidates for ridge regression, especially for large-scale problems.

\begin{table}[h]
\centering
\caption{Section 5.3: Test of the regression problem on the matrices from the University of Florida Sparse Matrix Collection, $tol=1e-3$ and $\tau=0.1$. The sampling ratio is chosen as $\eta=0.01$ in Algorithm \ref{alg10}, and the relaxation parameter $\omega=1+\frac{n}{m}$ in AVGRK~($\omega$).}
{\footnotesize\begin{tabular}{|c|c|c|c|c|c|c|}
\hline
\multicolumn{2}{|c|}{$m\times n,~\tau=0.1$}   & \makecell[c]{$ch7-8-b2$\\$11760\times1760$} & \makecell[c]{$ch7-9-b2$\\$17640\times1512$} & \makecell[c]{$ch8-8-b2$\\$18816\times1568$} & \makecell[c]{$ch6-6-b2$\\$2400\times450$} & \makecell[c]{$bcsstm09$\\$1083\times1083$} \\ \hline
\multirow{2}{*}{VRK}  & IT  & /    & /      & /    & /      & 6751    \\ \cline{2-7}
                      & CPU &/ &/    &/   & /    & 5.93   \\ \hline
\multirow{2}{*}{VGRK}  & IT  & /    & /      & /    & /      & 1130    \\ \cline{2-7}
                      & CPU &/ &/    &/   & /    & 1.77   \\ \hline
\multirow{2}{*}{VRGS}  & IT  & /    & /      & /    & /      & 7285    \\ \cline{2-7}
                      & CPU &/ &/    &/   & /    & 3.64   \\ \hline
\multirow{2}{*}{AVGRKRP~($\omega$)}  & IT  & 11051    & 13486      & 14107    & 4937      & 8     \\ \cline{2-7}
                      & CPU &871.81 &1940.11    &2015.63    & 38.03    & {\bf0.18}   \\ \hline
\multirow{2}{*}{\bf Algorithm 6} & IT  & 1314     & 1676     & 1739      & 557      & 25301     \\ \cline{2-7}
                      & CPU & 1.91   & {\bf3.42}    & 3.66  & 3.79    & 0.19    \\ \hline
\multirow{2}{*}{\bf Algorithm 7} & IT  & 1588     & 1682      & 1649     & 604      & 29431     \\ \cline{2-7}
                      & CPU & {\bf1.65}   & 3.58    & {\bf3.15}    & {\bf0.18}    & 2.01   \\ \hline
\end{tabular}}\label{T57}
\end{table}

\begin{table}[H]
\centering
\caption{Test of the regression problem on the matrices from the University of Florida Sparse Matrix Collection, $tol=1e-3$ and $\tau=0.01$. The sampling ratio is chosen as $\eta=0.01$ in Algorithm \ref{alg10}, and the relaxation parameter $\omega=1+\frac{n}{m}$ in AVGRK~($\omega$).}
{\footnotesize\begin{tabular}{|c|c|c|c|c|c|c|}
\hline
\multicolumn{2}{|c|}{$m\times n, \tau=0.01$}   & \makecell[c]{$ch7-8-b2$\\$11760\times1760$} & \makecell[c]{$ch7-9-b2$\\$17640\times1512$} & \makecell[c]{$ch8-8-b2$\\$18, 816\times1568$} & \makecell[c]{$ch6-6-b2$\\$2400\times450$} & \makecell[c]{$bcsstm09$\\$1083\times1083$} \\ \hline
\multirow{2}{*}{VRK}  & IT  & /    & /      & /    & /      & 7933    \\ \cline{2-7}
                      & CPU &/ &/    &/   & /    & 8.31   \\ \hline
\multirow{2}{*}{VGRK}  & IT  & /    & /      & /    & /      & 1084    \\ \cline{2-7}
                      & CPU &/ &/    &/   & /    & 1.71   \\ \hline
\multirow{2}{*}{VRGS}  & IT  & /    & /      & /    & /      & 7180    \\ \cline{2-7}
                      & CPU &/ &/    &/   & /    & 3.59   \\ \hline
\multirow{2}{*}{AVGRKRP~($\omega$)}  & IT  & 104573    & /      & /    & 48936     & 8     \\ \cline{2-7}
                      & CPU & 7794.59    &/    &/    & 397.24    & {\bf0.18}   \\ \hline
\multirow{2}{*}{\bf Algorithm 6} & IT  & 1314     & 1676     & 1739      & 557      & 25301     \\ \cline{2-7}
                      & CPU & {\bf 1.87}   & {\bf3.21}    & {\bf3.53}  & {\bf0.18}    & 0.86    \\ \hline
\multirow{2}{*}{\bf Algorithm 7} & IT  & 1513     & 1712      & 1839     & 579      & 30142      \\ \cline{2-7}
                      & CPU & 2.14   & 3.46    & 3.96    & 0.22    & 1.91   \\ \hline
\end{tabular}}\label{T58}
\end{table}

\begin{table}[H]
\centering
\caption{Test of the regression problem on the matrices from the University of Florida Sparse Matrix Collection, $tol=1e-3$ and $\tau=0.001$. The sampling ratio is chosen as $\eta=0.01$ in Algorithm \ref{alg10}, and the relaxation parameter $\omega=1+\frac{n}{m}$ in AVGRK~($\omega$).}
{\footnotesize\begin{tabular}{|c|c|c|c|c|c|c|}
\hline
\multicolumn{2}{|c|}{$m\times n, \tau=0.001$}   & \makecell[c]{$ch7-8-b2$\\$11760\times1760$} & \makecell[c]{$ch7-9-b2$\\$17640\times1512$} & \makecell[c]{$ch8-8-b2$\\$18816\times1568$} & \makecell[c]{$ch6-6-b2$\\$2400\times450$} & \makecell[c]{$bcsstm09$\\$1083\times1083$} \\ \hline
\multirow{2}{*}{VRK}  & IT  & /    & /      & /    & /      & 8306    \\ \cline{2-7}
                      & CPU &/ &/    &/   & /    & 8.97   \\ \hline
\multirow{2}{*}{VGRK}  & IT  & /    & /      & /    & /      & 1079    \\ \cline{2-7}
                      & CPU &/ &/    &/   & /    & 1.84   \\ \hline
\multirow{2}{*}{VRGS}  & IT  & /    & /      & /    & /      & 8210    \\ \cline{2-7}
                      & CPU &/ &/    &/   & /    & 4.83   \\ \hline
\multirow{2}{*}{AVGRKRP~($\omega$)}  & IT  & /    & /      & /    & /      & 9     \\ \cline{2-7}
                      & CPU & /    &/    &/    & /    & {\bf 0.19}   \\ \hline
\multirow{2}{*}{\bf Algorithm 6} & IT  & 1314     & 1676     & 1739      & 557      & 1083     \\ \cline{2-7}
                      & CPU & {\bf1.90}  & {\bf3.23}    & 3.44  & {\bf0.19}    & 0.21    \\ \hline
\multirow{2}{*}{\bf Algorithm 7} & IT  & 1308     & 1851      & 1489     & 669      & 29987      \\ \cline{2-7}
                      & CPU & 2.21   & 3.37    & {\bf3.41} & 0.67    & 2.02  \\ \hline
\end{tabular}}\label{T59}
\end{table}

\section{Concluding Remarks}
Kaczmarz method is an effectively iterative method for large linear systems. The key of this method is to introduce a practical and suitable probability criterion for selecting working rows from the coefficient matrix. In this paper, we propose a new probability criterion which can capture as large items as possible in the homogenized residual of linear systems in each iteration, and accelerate the algorithm by increasing the probability saliency and random sampling. This method converges faster than RK method both in theory and in practice, and it often converges much faster than the GRK method for large-scale problems.

First, from the probability significance point of view, we present a partially randomized Kaczmarz method, which can reduce the computational overhead needed in greedy randomized Kaczmarz method.
Second, based on Chebyshev's law of large numbers and Z-test, we apply a simple sampling approach to the partially randomized Kaczmarz method, and propose a randomized Kaczmarz method with simple random sampling for large linear systems. The convergence of the proposed method is established.

\end{document}